\documentclass[10pt ]{article}
\usepackage{amsmath,amssymb,amsthm,mathrsfs,dsfont}
\usepackage[margin=3cm]{geometry}
\usepackage{titlesec,hyperref}
\usepackage{color}
\usepackage{fancyhdr}
\usepackage{subfigure}
\usepackage[graphicx]{realboxes}
\usepackage{float}
\usepackage{enumerate}
\usepackage{enumitem}
\setlist[enumerate]{label=(\arabic*), leftmargin=*, itemsep=2pt}
\usepackage{dsfont} 
\usepackage{extarrows}
\titleformat{\subsection}{\it}{\thesubsection.\enspace}{1pt}{}
\newtheorem{theorem}{Theorem}[section]

\newtheorem{corollary}{Corollary}[section]
\newtheorem{definition}{Definition}[section]
\newtheorem{lemma}{Lemma}[section]
\newtheorem{proposition}{Proposition}[section]
\newtheorem{remark}{Remark}[section]

\numberwithin{equation}{section}

\allowdisplaybreaks

\numberwithin{equation}{section}

\begin{document}
\title{Local well-posedness, blow-up phenomenon and ill-posedness for a model system of the radiating gas
	\hspace{-4mm}
}

\author{
	Jian $\mbox{Chen}$ \footnote{E-mail:  chenj883@mail2.sysu.edu.cn}, \quad
    Zhaonan $\mbox{Luo}$ \footnote{E-mail:  luozhn7@mail.sysu.edu.cn}, \quad
	Yuchen $\mbox{Yang}$\footnote{E-mail: yangych39@mail2.sysu.edu.cn} \quad and
	Zhaoyang $\mbox{Yin}$\footnote{E-mail: mcsyzy@mail.sysu.edu.cn}\\
	$\mbox{School}$ of Science,\\ {Shenzhen Campus of Sun Yat-sen University, Shenzhen 518107, China}
}
\date{}
\maketitle
\hrule

\begin{abstract}
	 In this paper, we investigate  the Cauchy problem for a model system of radiating gas in multi-dimensional space. Firstly, we prove that the model
	is locally well-posed in Besov spaces $B^s_{p,r}(\mathbb{R}^d)$ with $s>1+\frac{d}{p}$
(resp., $s=1+\frac{d}{p}$, $r=1$), including the endpoint space $B^1_{\infty,1}(\mathbb{R}^d)$. Secondly, we show that the data-to-solution map is not uniformly continuous on bounded subsets of these Besov spaces. Moreover, we construct a blow-up
	scenario and derive two sufficient blow-up conditions, which improve the previous blow-up condition obtained in
	\cite{MR3016436}.
	Finally, we prove
	that the Cauchy problem is ill-posed in the critical space
	$B^{1+\frac{d}{p}}_{p,r}(\mathbb{R}^d)$ for $1\leq p\leq\infty$ and $r>1$ in the sense of norm
	inflation.

	\vspace*{5pt}
					
	\vspace*{5pt}
	
\end{abstract}

\vspace*{10pt}

\tableofcontents

\section{Introduction}

This paper is concerned with the local well-posedness, blow-up phenomenon and the ill-posedness of the Cauchy problem of the hyperbolic-elliptic coupled system, taking the form of
\begin{align}\label{eq}
    \begin{cases}
     \partial_tu+\nabla\cdot f(u)+ \mathrm{div}q = 0, &\ x\in\mathbb{R}^d,t>0,\\
   -\nabla \mathrm{div}q+q+\nabla u=0, &\ x\in \mathbb{R}^d,t>0,   \\
   u(x,t)|_{t=0} = u_0(x), &\ x\in\mathbb{R}^d,
    \end{cases}
\end{align}
where $u = u(x,t):\mathbb{R}^d \times[0,\infty)\rightarrow\mathbb{R},\ q=\left(q_1(x,t),\dots,q_d(x,t)\right):\mathbb{R}^d \times[0,\infty)\rightarrow\mathbb{R}^d$ are unknown functions. $d \geq 1$ denotes the spatial dimension. The function $f(u)=\left(f_1(u),\dots,f_d(u)\right):\mathbb{R}\rightarrow\mathbb{R}^d$ is a given smooth flux function. For computational convenience, we express $q$ in terms of $u$, which reduces the system (\ref{eq}) to
\begin{align}\label{eq;scalar}
    \begin{cases}
        \partial_tu+\alpha(u)\cdot\nabla u = H[u], &\ x\in\mathbb{R}^d,t>0,\\
        u(x,t)|_{t=0} = u_0(x), &\ x\in\mathbb{R}^d,
    \end{cases}
\end{align}
where $\alpha(u):=(f'_{1}(u),\dots,f'_{d}(u))$, $H[u] := (I-\Delta)^{-1}\Delta u$ and $\Delta:= \overset{d}{\underset{i=1}{\sum}}
 \frac{\partial^2}{\partial x_i^2}$.
System (\ref{eq}) was first derived in \cite{hamer1971} for one dimensional version and in \cite{diFrancesco2007} for multi-dimensional version, which simplifies the model for the motion of radiating gas in $\mathbb{R}^d$. It provides a simplified description of the interaction between nonlinear convection and radiative heat transport in a compressible gas. In a certain physical situation, system (\ref{eq}) gives a good approximation to the fundamental system of equations describing the motion of a radiating gas
\begin{align}\label{eq;ra}
\begin{cases}
\rho_t + \operatorname{div}(\rho u) = 0,\\
(\rho u)_t + \operatorname{div}\big(\rho u\otimes u + pI\big) = 0,\\
\left\{\rho\left(e+\frac{|u|^2}{2}\right)\right\}_t
+ \operatorname{div}\left\{\rho u\left(e+\frac{|u|^2}{2}\right)+pu+q\right\}=0,\\
-\nabla\operatorname{div}q + a q + b\nabla(\theta^4)=0,
\end{cases}
\end{align}
where $\rho \geq0,u,p,e$ and $\theta \geq 0$ denote the mass density, velocity, pressure, internal energy and absolute temperature of the gas, respectively. $q$ is the radiative heat flux, and $a,b$ are two given positive constants depending on the physical properties of the gas. 

As an important model for radiating gases, system \eqref{eq} has attracted considerable attention since its introduction and has been extensively studied. In the one-dimensional case, Schochet and Tadmor \cite{Schochet199295} first studied the existence of traveling waves. Kawashima and Nishibata \cite{Kawashima199895} investigated the existence and asymptotic stability of traveling waves. Lattanzio and Marcati \cite{MR1970037} studied the global well-posedness and relaxation limits of the model. Ji, Zhang and Zhu \cite{doi:10.1137/21M1459848} established the asymptotic stability of solutions on the half-line. Moreover, Hong, Zhang and Zhu \cite{HONG2026111011} studied the nonlinear global stability of rarefaction waves under large perturbations. For a related parabolic–elliptic system, Zhu and Zhu \cite{Zhu2026} proved the unconditional global stability of solutions on the half-line. For further results in the one-dimensional case, we refer the reader to \cite{MR2271601,MR1671523, MR2579710, MR3010836, MR1979850, MR2812581}. In the multi-dimensional case, Wang and Wang \cite{MR2527538} studied the pointwise estimates of solutions. They also investigated the blow-up and global existence of classical solutions \cite{MR3016436}. Gao, Ruan and Zhu \cite{MR2414408} obtained decay rates to the planar rarefaction waves for the model system. Zhang and Zhu \cite{MR4911714} showed the asymptotic stability of the planar rarefaction wave in a two-dimensional half-space. Recently, Zhang, Zhu and Zhu \cite{MR4972319} proved unconditional global stability of solutions to the Cauchy problem for a related parabolic-elliptic system. For further developments in the multi-dimensional case, we refer the reader to \cite{MR2402882,MR3224577,MR2853190,MR3264733,MR4768460,MR4623809}.

Recently, Crin-Barat and Said-Houari \cite{crinbarat2026hamer} established global well-posedness for system (\ref{eq}) with small initial data in hybrid Besov spaces with distinct regularity exponents at low and high frequencies. They also derived optimal time decay estimates under the additional assumption $u_0 \in \dot{B}^{-d/2}_{2,\infty}(\mathbb{R}^d)$ for $d\geq2$. However, for general $1\le p\leq \infty$, the local well-posedness has not been studied. In this paper, we establish local well-posedness in general Besov spaces, with particular attention to the critical spaces and the endpoint space $B_{\infty,1}^{1}(\mathbb{R}^d)$. Local well-posedness in the critical Besov spaces with $p<\infty$ follows from standard arguments (see \cite{MR4595933}). In the endpoint space $B^1_{\infty,1}(\mathbb{R}^d)$, existence is obtained by a compactness argument, while uniqueness is proved by combining a logarithmic interpolation inequality with Osgood's lemma. Following the idea in \cite{MR3926133}, we derive $L^\infty$ and $B^1_{\infty,1}$ estimates for the difference of two solutions and use a Bona-Smith type argument to establish continuous dependence on the initial data. We further show that the data-to-solution map $u_0\mapsto S_t(u_0)$ is not uniformly continuous on bounded subsets of $B^s_{p,r}(\mathbb{R}^d)$ for the range of indices covered by our local well-posedness result.

In addition, we establish a blow-up scenario in Besov spaces and derive two sufficient conditions for finite time blow-up of solutions to the multi-dimensional scalar equation (\ref{eq;scalar}). Our first approach is to trace the dynamics of the directional derivative along the characteristics. By establishing a maximum principle and estimating the nonlocal term in terms of the initial oscillation, we derive a sufficient condition for finite time blow-up. The other approach we adopt is inspired by the pioneering work of Constantin and Escher \cite{constantin1998wave}. To this end, we generalize their result on the evolution of extrema to higher dimensions. Then by analyzing the Riccati-type differential inequalities for the extremal values of $\underset{i=1}{\overset{d}{\sum}}\partial_{x_i}u$ and using the decay estimate for the upper extremum (see \eqref{zyM}), we obtain a new blow-up criterion. Our blow-up results improve and extend the earlier work of Wang and Wang \cite{MR3016436}.

Finally, to study the sharpness of our well-posedness result in Besov spaces, we focus on whether this system is well-posed in $B^{1+\frac{d}{p}}_{p,r}(\mathbb{R}^d)$ with $1\leq p\leq\infty$ and $r>1$. Inspired by the works \cite{MR4413296,MR4888641,MR3906228}, we construct a family of special initial data $\{u_0^N\}$ such that each $u_0^N$ satisfies our blow-up condition and $\|u_0^N\|_{B^{1+\frac{d}{p}}_{p,r}}\to0$ as $N\to\infty$, which leads to norm inflation in arbitrarily short time.

The main results of this paper are stated as follows:

\begin{theorem}\label{thm;1}
    Let $s>1+\frac{d}{p}$, $1\leq p,r \leq \infty$ or $s = 1+ \frac{d}{p}$, $r = 1$ with $1 \leq p \leq \infty$. Assume $\alpha$ is a smooth function, $\alpha(0)=0$ and $u_0 \in B^s_{p,r}$, then there exists a time $T>0$ such that the  problem (\ref{eq;scalar}) is well-posed in $E^s_{p,r}$ in the sense of Hadamard, where $E^s_{p,r}$ is defined as
    \begin{align*}
E_{p,r}^s(T) \triangleq
\begin{cases}
C\bigl([0,T]; B_{p,r}^s\bigr) \cap {C}^1\bigl([0,T]; B_{p,r}^{s-1}\bigr)
& \text{if } r < \infty,
\\[6pt]
{C}_w\bigl([0,T]; B_{p,\infty}^s\bigr) \cap {C}^{0,1}\bigl([0,T]; B_{p,\infty}^{s-1}\bigr)
& \text{if } r = \infty.
\end{cases}
\end{align*}
    
\end{theorem}

\begin{theorem}\label{thm;un}
    Let $(s,p,r)$ satisfy the condition
    \begin{align}\label{cond;un}
       s>1+\frac{d}{p}, 1\leq p,r \leq \infty \ \text{or}\ s = 1+ \frac{d}{p}, r = 1 \ \text{with}\ 1\leq p \leq \infty.
    \end{align}
    Assume $\alpha$ is a smooth function, $\alpha(0)=0$ and $\alpha\not\equiv
0$. Then the data-to-solution map $u_0\mapsto S_t(u_0)$ for the Cauchy problem (\ref{eq;scalar}) is not uniformly continuous on bounded subsets of
$B^s_{p,r}(\mathbb{R}^d)$. More precisely, there exist $T_u>0$, $f_n$, $g_n$and $\kappa_0>0$ such that the
sequences satisfy
\begin{equation}\label{eq:initsep}
\|{(f_n+g_n)-f_n}\|_{B^s_{p,r}}=\|{g_n}\|_{B^s_{p,r}}\longrightarrow0,
\end{equation}
while
\begin{equation}\label{eq:timesep}
\lim_{n\to\infty}\inf\|{S_t(f_n+g_n)-S_t(f_n)}\|_{B^s_{p,r}}\ge\kappa_0\,t
\qquad\text{for every }t\in[0,T_u].
\end{equation}
\end{theorem}

In the analysis of blow-up behaviors, we impose the following assumptions that $f_1=\dots=f_d$ with $f''_i\geq\alpha_0>0$.

\begin{theorem}\label{thm;blowup}
Consider the Cauchy problem (\ref{eq;scalar}) with initial data $u_0 \in H^s(\mathbb{R}^d)$ with $s>1+\frac{d}{2}$. Let $T$ be the maximal existence time of the unique solution $u$. Assume there exists a point $x_0 \in \mathbb{R}^d$ such that
    \begin{align}\label{cond;1.4}
        \sum_{i=1}^{d}\partial_{x_i}u_0(x_0)< \frac{-1-\sqrt{1+4\alpha_0k}}{2\alpha_0},
    \end{align}
where $\delta = \underset{x\in \mathbb{R}^d}{\sup}u_0(x)-\underset{x\in \mathbb{R}^d}{\inf}u_0(x) $,   $k:=\min\{\frac{d}{2}\delta,b_0\}$ and $b_0 := \underset{x\in \mathbb{R}^d}{\sup} \underset{i=1}{\overset{d}{\sum}}\partial_{x_i}u_0(x)$. Then the corresponding solution $u$ blows up in finite time with maximal existence time $T$. The blow-up time $T$ can be estimated as
 \begin{align*}
     T\leq \frac{\log R}{\sqrt{1+4\alpha_0k}}, \ \text{with} \ R:=\frac{-1+\sqrt{1+4\alpha_0k}-2\alpha_0\underset{i=1}{\overset{d}{\sum}}\partial_{x_i}u_0(x_0)}{-1-\sqrt{1+4\alpha_0k}-2\alpha_0\underset{i=1}{\overset{d}{\sum}}\partial_{x_i}u_0(x_0)}.
 \end{align*}
 \end{theorem}
 \begin{remark}
     Theorem 1.1 in \cite{MR3016436} provides a sufficient condition for finite-time blow-up of the solution to the Cauchy problem, expressed in terms of the initial data $u_0$:
     \begin{align}\label{cond;2}
         \sum_{i=1}^d \partial_{x_i} u_0(x_0) < \omega_*,
     \end{align}
     where
     \begin{align*}
         \omega_* = \frac{-1-\sqrt{1+4\alpha_0\max\{0,B_1\}}}{2\alpha_0},\  B_1=\sup_{x\in\mathbb{R}^d}\sum_{i=1}^d \partial_{x_i} u_0(x).
     \end{align*}
     Comparing (\ref{cond;1.4}) with (\ref{cond;2}) directly, it is easy to see that
     \begin{align*}
         \omega_* \leq \frac{-1-\sqrt{1+4\alpha_0k}}{2\alpha_0}.
     \end{align*}
     Hence, our blow-up result improves the corresponding result in \cite{MR3016436}.
 \end{remark}

\begin{theorem}\label{thm;Mm}
    Consider the Cauchy problem (\ref{eq;scalar}) with initial data $u_0\in H^s(\mathbb R^d)$ with $s>1+\frac d2$.
Let $T$ be the maximal existence time of the unique solution $u$. Define $m(t):=\underset{x\in\mathbb R^d}{\inf}\underset{i=1}{\overset{d}{\sum}}u_{x_i}(t,x),\
M(t):=\underset{x\in\mathbb R^d}{\sup}\underset{i=1}{\overset{d}{\sum}}u_{x_i}(t,x).$ Assume that $m(0)$ and $M(0)$ satisfy
\begin{align}  \label{cond;bu}
m(0)<-\frac{3+\sqrt{1+8\alpha_0M(0)}}{4\alpha_0}.
\end{align}
Then the corresponding solution $u$ blows up in finite time with maximal existence time $T$.
The blow-up time $T$ can be estimated as
\begin{align*}
T\leq
\log \left(
1-\frac{4}{4\alpha_0 m(0)+3+\sqrt{1+8\alpha_0 M(0)}}\right)
.
\end{align*}
\end{theorem}

\begin{remark}
Fix \(\alpha_0>0\),
\(d\ge2\), and \(h>0\), and consider
\begin{align*}
    u_0(x)
    =
    \frac{7h}{4\alpha_0}
    \sqrt{\frac{\mathrm{e}}{2d}}\,
    \mathrm{e}^{-|x|^2/h^2},
    \qquad x\in\mathbb{R}^d.
\end{align*}
Then \(u_0\in\mathcal{S}(\mathbb{R}^d)\subset H^s(\mathbb{R}^d)\), and
\begin{align*}
    \sum_{i=1}^{d}\partial_{x_i}u_0(x)
    =
    -\frac{7}{2\alpha_0h}
    \sqrt{\frac{\mathrm{e}}{2d}}\,
    \left(\sum_{i=1}^{d}x_i\right)\mathrm{e}^{-|x|^2/h^2}.
\end{align*}
Using the fact that
\begin{align*}
\left|\sum_{i=1}^{d}x_i\right|
\mathrm{e}^{-|x|^2/h^2}
\le
\sqrt{d}\,|x|\mathrm{e}^{-|x|^2/h^2}
\le
h\sqrt{\frac{d}{2\mathrm{e}}},
\end{align*}
we obtain
\begin{align*}
    m(0)=-\frac{7}{4\alpha_0},
    \qquad
    M(0)=\frac{7}{4\alpha_0}.
\end{align*}
Consequently,
\begin{align*}
    m(0)
    =
    -\frac{7}{4\alpha_0}
    <
    -\frac{3+\sqrt{1+8\alpha_0M(0)}}{4\alpha_0}
    .
\end{align*}
 Hence the blow-up condition in Theorem \ref{thm;Mm} is satisfied.

On the other hand, for the criterion \eqref{cond;2}, we have
\begin{align*}
    B_1=M(0)=\frac{7}{4\alpha_0},\
    \omega_*
    =
    -\frac{1+2\sqrt{2}}{2\alpha_0}.
\end{align*}
Since
\begin{align*}
-\frac{7}{4\alpha_0}
>
-\frac{1+2\sqrt{2}}{2\alpha_0}
=\omega_*,
\end{align*}
it follows that
\begin{align*}
\sum_{i=1}^{d}\partial_{x_i}u_0(x)>\omega_*
\qquad
\text{for all }x\in\mathbb{R}^d.
\end{align*}
Therefore, condition (\ref{cond;2}) fails everywhere. This example shows that our blow-up criterion enlarges the class of initial data covered by the result in
\cite{MR3016436}. 
\end{remark}

\begin{theorem}\label{thm;6}
Let  $1\leq p\leq\infty$, $1<r\leq\infty$, and assume that
$f_1=\dots=f_d$ with $f_i''\geq\alpha_0>0$. Then the Cauchy problem
\eqref{eq;scalar} is ill-posed in $B^{1+\frac{d}{p}}_{p,r}(\mathbb{R}^d)$ in the sense
of norm inflation: for any $N\in\mathbb{N}$ large enough there exists
$u_0\in H^{\infty}(\mathbb{R}^d)\cap B^{1+\frac{d}{p}}_{p,r}(\mathbb{R}^d)$ such that the following hold:
\begin{enumerate}
    \item[(1)] $\displaystyle
    \|u_0\|_{B^{1+\frac{d}{p}}_{p,r}}\leq\frac{1}{\log N}
    \;\longrightarrow\;0$ \quad as $N\to\infty$;
    \item[(2)] the corresponding solution
    $u\in C\big([0,T);H^{\infty}(\mathbb{R}^d)\big)$ of \eqref{eq;scalar} is
    unique, and its maximal existence time $T$ satisfies
    \begin{align*}
        0<T\leq\frac{1}{\log N};
    \end{align*}
    \item[(3)] $\displaystyle
    \|u\|_{L^{\infty}([0,T);B^{1+\frac{d}{p}}_{p,r})}=+\infty$, i.e.\
    $\underset{t\rightarrow T}{\lim}\sup\|u(t)\|_{B^{1+\frac{d}{p}}_{p,r}}=+\infty$; in particular,
    for every $N$ large enough there exists $t_N\in(0,T)$, hence
    $t_N<T\leq\frac{1}{\log N}$, with
    $\|u(t_N)\|_{B^{1+\frac{d}{p}}_{p,r}}\geq\log N$.
\end{enumerate}

\end{theorem}
The paper is organized as follows. In Section \ref{Preliminaries}, we recall some preliminary results that will be used throughout the paper. In Section \ref{Local well-posedness}, we establish local well-posedness in $B^s_{p,r}(\mathbb{R}^d)$ for $s>1+\frac{d}{p}$ with $1\leq p,r\leq\infty$, or $s=1+\frac{d}{p}$ with $r=1$ and $1\leq p\leq\infty$. In Section \ref{Nonuniform continuous dependence}, we prove the non-uniform dependence on initial data for (\ref{eq;scalar}) in $B^s_{p,r}$ with $s>1+\frac{d}{p}$, $1\leq p,r \leq \infty$ or $s = 1+ \frac{d}{p}$, $r = 1$, $1 \leq p \leq \infty$. In Section \ref{Blow-up}, we derive a blow-up criterion and obtain two sufficient conditions for finite time blow-up. In Section \ref{Ill-posedness}, we show that the Cauchy problem \eqref{eq;scalar} is ill-posed in the critical Besov spaces $B^{1+\frac{d}{p}}_{p,r}(\mathbb{R}^d)$ with $1\leq p\leq\infty$ and $1<r\leq\infty$ in the sense of norm inflation.

\section{Preliminaries}\label{Preliminaries}

\begin{proposition}\label{prop2.1}
(\cite{bahouri2011fourier}) Let $\mathcal{C}$ be an annulus and $\mathcal{B}$ a ball. A constant $C$ exists such that for any $k\in\mathbb{N}$, $1\leq p\leq q\leq\infty$, and any function $u$ of $L^p(\mathbb{R}^d)$, we have
\begin{align*}
\operatorname{Supp}\, \widehat{u}\subset\lambda\mathcal{B}
\implies
\|D^k u\|_{L^q}
= \sup_{|\alpha|=k} \|\partial^\alpha u\|_{L^q}
\leq C^{k+1}\lambda^{k+d\left(\frac{1}{p}-\frac{1}{q}\right)}\|u\|_{L^p},
\end{align*}
\begin{align*}
\operatorname{Supp}\, \widehat{u}\subset\lambda\mathcal{C}
\implies
C^{-k-1}\lambda^k\|u\|_{L^p}
\leq \|D^k u\|_{L^p}
\leq C^{k+1}\lambda^k\|u\|_{L^p}.
\end{align*}
\end{proposition}

\begin{proposition}
    (\cite{bahouri2011fourier}) Let $\mathcal{B} = \{\xi \in \mathbb{R}^d : |\xi| \leq \tfrac{4}{3}\}$ and $\mathcal{C} = \{\xi \in \mathbb{R}^d : \tfrac{3}{4} \leq |\xi| \leq \tfrac{8}{3}\}$. There exist two smooth, radial functions $\chi$ and $\varphi$, valued in the interval $[0,1]$, belonging respectively to $\mathcal{D}(\mathcal{B})$ and $\mathcal{D}(\mathcal{C})$, such that
\begin{align*}
\forall \xi \in \mathbb{R}^d,\ \chi(\xi) + \sum_{j\geq 0} \varphi(2^{-j}\xi) = 1,
\end{align*}
\begin{align*}
\forall \xi \in \mathbb{R}^d\setminus\{0\},\ \sum_{j\in\mathbb{Z}} \varphi(2^{-j}\xi) = 1,
\end{align*}
\begin{align*}
|j-j'| \geq 2 \implies \operatorname{Supp}\, \varphi(2^{-j}\cdot) \cap \operatorname{Supp}\, \varphi(2^{-j'}\cdot) = \emptyset,
\end{align*}
\begin{align*}
j \geq 1 \implies \operatorname{Supp}\, \chi \cap \operatorname{Supp}\, \varphi(2^{-j}\cdot) = \emptyset,
\end{align*}
the set $\widetilde{\mathcal{C}} \triangleq \mathcal{B}(0,\tfrac{2}{3}) + \mathcal{C}$ is an annulus, and we see
\begin{align*}
|j-j'| \geq 5 \implies 2^j \widetilde{\mathcal{C}} \cap 2^{j'} \mathcal{C} = \emptyset.
\end{align*}

\noindent Let $u \in \mathcal{S}'$. Defining
\begin{align*}
\Delta_j u \triangleq 0 \text{ if } j \leq -2,\ \Delta_{-1}u \triangleq \chi(D)u = \mathcal{F}^{-1}(\chi \mathcal{F}u),
\end{align*}
\begin{align*}
\Delta_j u \triangleq \varphi(2^{-j}D)u = \mathcal{F}^{-1}(\varphi(2^{-j}\cdot)\mathcal{F}u) \text{ if } j \geq 0,
\end{align*}
\begin{align*}
S_j u \triangleq \sum_{j' \leq j-1} \Delta_{j'} u,
\end{align*}
we have the following Littlewood-Paley decomposition:
\begin{align*}
u = \sum_{j\in\mathbb{Z}} \Delta_j u \quad \text{in } \mathcal{S}'.
\end{align*}
\end{proposition}

\begin{definition}

     Let $s \in \mathbb{R}$ and $(p,r) \in [1,\infty]^2$. The nonhomogeneous Besov space $B^s_{p,r}$ consists of all $u \in \mathcal{S}'$ such that
\begin{align*}
\|u\|_{B^s_{p,r}} \triangleq \left\| \left(2^{js} \|\Delta_j u\|_{L^p}\right)_{j\in\mathbb{Z}} \right\|_{\ell^r(\mathbb{Z})} < \infty.
\end{align*}
\end{definition}

\begin{definition}
    Considering $u, v \in \mathcal{S}'$, we have the following Bony decomposition:
\begin{align*}
uv = T_u v + T_v u + R(u,v),
\end{align*}
where
\begin{align*}
T_u v = \sum_j S_{j-1}u \Delta_j v,\quad R(u,v) = \sum_{|k-j|\leq 1} \Delta_k u \Delta_j v.
\end{align*}
\end{definition}

\begin{proposition}\label{prop;3}
(\cite{bahouri2011fourier}) Let $s, s' \in \mathbb{R}$, $1 \leq p,p_1,p_2,r,r_1,r_2 \leq \infty$.
\begin{enumerate}
    \item $B^s_{p,r}$ is a Banach space, and is continuously embedded in $\mathcal{S}'$.
    \item If $r<\infty$, then $\underset{j\to\infty}{\lim}
\|S_j u - u\|_{B^s_{p,r}} = 0$. If $p,r<\infty$, then $\mathcal{C}_0^\infty$ is dense in $B^s_{p,r}$.
    \item If $p_1 \leq p_2$ and $r_1 \leq r_2$, then $B^s_{p_1,r_1} \hookrightarrow B^{s-d\left(\frac{1}{p_1}-\frac{1}{p_2}\right)}_{p_2,r_2}$. If $s<s'$, then the embedding $B^{s'}_{p,r_2} \hookrightarrow B^s_{p,r_1}$ is locally compact.
    \item $B^s_{p,r} \hookrightarrow L^\infty \iff s>\frac{d}{p}$ or $s=\frac{d}{p},\ r=1$.
    \item Fatou property: if $(u_n)_{n\in\mathbb{N}}$ is a bounded sequence in $B^s_{p,r}$, then an element $u\in B^s_{p,r}$ and a subsequence $(u_{n_k})_{k\in\mathbb{N}}$ exist such that
    \begin{align*}
    \underset{k\to\infty}{\lim} u_{n_k} = u \text{ in } \mathcal{S}' \quad \text{and} \quad \|u\|_{B^s_{p,r}} \leq C\lim_{k\to\infty}\inf\|u_{n_k}\|_{B^s_{p,r}}.
    \end{align*}
    \item Let $m\in\mathbb{R}$ and $f$ be a $S^m$-multiplier, i.e. $f$ is a smooth function and satisfies that $\forall \alpha\in\mathbb{N}^d$, $\exists C=C(\alpha)$ such that $|\partial^\alpha f(\xi)| \leq C(1+|\xi|)^{m-|\alpha|},\ \forall \xi\in\mathbb{R}^d$. Then the operator $f(D)=\mathcal{F}^{-1}(f\mathcal{F}\cdot)$ is continuous from $B^s_{p,r}$ to $B^{s-m}_{p,r}$.
    \item If $s<s'$, then for all $\phi\in\mathcal{S}(\mathbb{R}^d)$, multiplication by $\phi$ is a compact operator from $B^{s'}_{p,\infty}$ to $B^s_{p,1}$.
\end{enumerate}
\end{proposition}

\begin{lemma}\label{lem;1}
(\cite{bahouri2011fourier})The following estimates hold:
\begin{enumerate}
    \item For any $s>0$ and any $(p,r)\in[1,\infty]^2$, the space $L^\infty\cap B^s_{p,r}$ is an algebra, and a constant $C=C(s,d)$ exists such that
    \begin{align*}
    \|uv\|_{B^s_{p,r}} \leq C\bigl(\|u\|_{L^\infty}\|v\|_{B^s_{p,r}} + \|u\|_{B^s_{p,r}}\|v\|_{L^\infty}\bigr).
    \end{align*}
    \item If $1\leq p,r\leq\infty$, $s_1\leq s_2$, $s_2>\frac{d}{p}$ ($s_2\geq\frac{d}{p}$ if $r=1$) and $s_1+s_2>\max\bigl\{0,\frac{2d}{p}-d\bigr\}$, there exists $C=C(s_1,s_2,p,r,d)$ such that
    \begin{align*}
    \|uv\|_{B^{s_1}_{p,r}} \leq C\|u\|_{B^{s_1}_{p,r}}\|v\|_{B^{s_2}_{p,r}}.
    \end{align*}
\end{enumerate}
\end{lemma}

\begin{lemma}\label{lem;2}
(\cite{bahouri2011fourier})
Consider the following class of transport equations:
\begin{align}\label{eq;trans}
\begin{cases} 
\partial_t f + v \cdot \nabla f + A  f = g \\ 
f|_{t=0} = f_0, 
\end{cases}
\end{align}
Let $1 \leq p \leq p_1 \leq \infty$, $1 \leq r \leq \infty$. Assume that
\begin{align}\label{cond;3.14}
\sigma \geq -d\min\left(\frac 1{p_1},\frac 1{p'}\right)
\quad\text{or}\quad
\sigma \geq -1 -d\min\left(\frac 1{p_1},\frac 1{p'}\right)
\quad\text{if } \operatorname{div}v=0  
\end{align}
with strict inequality if $r<\infty$.

There exists a constant $C$, depending only on $d,p,p_1,r$, and $\sigma$, such that
for all solutions $f\in L^\infty\big([0,T];B_{p,r}^\sigma\big)$ of (\ref{eq;trans}) with $A\equiv0$, initial data $f_0$ in $B_{p,r}^\sigma$,
and $g$ in $L^1\big([0,T];B_{p,r}^\sigma\big)$, we have, for a.e. $t\in[0,T]$,
\begin{align}\label{syfs}
\|f\|_{\widetilde{L}_t^\infty(B_{p,r}^\sigma)}
&\leq \bigg(\|f_0\|_{B_{p,r}^\sigma}
+\int_0^t \exp\big(-C V_{p_1}(t')\big)\|g(t')\|_{B_{p,r}^\sigma}dt'\bigg)\exp\big(C V_{p_1}(t)\big)  
\end{align}
with, if the inequality is strict in (\ref{cond;3.14}),
\begin{align}\label{3.14;est}
V_{p_1}'(t) :=
\begin{cases}
\|\nabla v(t)\|_{B_{p_1,\infty}^{\frac{d}{p_1}}\cap L^\infty},
& \text{if } \sigma < 1+\dfrac{d}{p_1},\\[4pt]
\|\nabla v(t)\|_{B_{p_1,1}^{\sigma-1}},
& \text{if } \sigma >1+\dfrac{d}{p_1} \;\text{or}\;
\big\{\sigma=1+\dfrac{d}{p_1} \;\text{and}\; r=1\big\},
\end{cases}
\end{align}
and, if equality holds in (\ref{cond;3.14}) and $r=\infty$,
\begin{align*}
V_{p_1}'(t) \overset{\mathrm{def}}{=} \|\nabla v(t)\|_{B_{p_1,1}^{\frac{d}{p_1}}}.
\end{align*}

If $f=v$, then for all $\sigma>0$ ($\sigma>-1$, if $\operatorname{div}v=0$), the estimate (\ref{syfs}) holds with
\begin{align*}
V_{p_1}'(t)=\|\nabla v(t)\|_{L^\infty}.
\end{align*}
\end{lemma}

\begin{lemma}\label{lem;3}
(\cite{MR3609245})   Let \( F : \mathbb{R} \to \mathbb{R} \) be smooth with \( F(0) = 0 \). For all \( 1 \leq p, r \leq \infty \) and \( s > 0 \), it holds that \( F(u) \in \dot{B}_{p,r}^s \cap L^\infty \) for \( u \in \dot{B}_{p,r}^s \cap L^\infty \), and
\begin{align*}
\|F(u)\|_{\dot{B}_{p,r}^s} \leq C \|u\|_{\dot{B}_{p,r}^s}
\end{align*}
with \( C \) depending only on \( \|u\|_{L^\infty} \), \( F' \) (and higher derivatives), \( s \), \( p \) and \( d \).

In the case \( s > -\min\left(\frac{d}{p}, \frac{d}{p'}\right) \) then \( u \in \dot{B}_{p,r}^s \cap \dot{B}_{p,1}^{\frac{d}{p}} \) implies that \( F(u) \in \dot{B}_{p,r}^s \cap \dot{B}_{p,1}^{\frac{d}{p}} \), and
\begin{align*}
\|F(u)\|_{\dot{B}_{p,r}^s} \leq C(1 + \|u\|_{\dot{B}_{p,1}^{\frac{d}{p}}}) \|u\|_{\dot{B}_{p,r}^s}.
\end{align*}
\end{lemma}

\begin{lemma}\label{lem;4}
(\cite{bahouri2011fourier,MR3552560})    Let $1 \leq p,\, r \leq \infty$ and $\theta>-\min \big\{\frac{d}{p}, \frac{d}{p'}\big\}$.
Let $f_{0} \in B_{p, r}^{\theta}$, $g \in L^{1}(0, T ; B_{p, r}^{\theta})$, and $v \in L^{\rho}(0, T ; B_{\infty, \infty}^{-M})$ for some $\rho>1$ and $M>0$ such that
\begin{align*}
\nabla v &\in L^{1}\big([0, T] ; B_{p, \infty}^{\frac{d}{p}} \cap L^{\infty}\big), \quad \text{if } \theta<1+\frac{d}{p},\\
\nabla v &\in L^{1}\big([0, T] ; B_{p, r}^{\theta}\big), \quad \text{if } \theta=1+\frac{d}{p},\ r>1,\\
\nabla v &\in L^{1}\big([0, T] ; B_{p, r}^{\theta-1}\big), \quad \text{if } \theta>1+\frac{d}{p}\big(\text{or } \theta=1+\frac{d}{p},\ r=1\big).
\end{align*}
Then the problem (\ref{eq;trans}) has a unique solution $f$ in
\[
\begin{cases}
C([0, T] ; B_{p, r}^{\theta}), & \text{if } r<\infty,\\
\big(\underset{\theta'<\theta}{\bigcap}
 C([0, T] ; B_{p, \infty}^{\theta'})\big) \cap C_{w}([0, T] ; B_{p, \infty}^{\theta}), & \text{if } r=\infty.
\end{cases}
\]
\end{lemma}

\begin{lemma}\label{lem;5}
(\cite{bahouri2011fourier,MR3552560})    Let $1 \leq p,\, r \leq\infty$ and $\theta>-\min\big\{\frac{d}{p},\frac{d}{p'}\big\}$. There exists a constant $C$ such that for all solutions $f \in L^{\infty}(0, T ; B_{p, r}^{\theta})$ of (\ref{eq;trans}) with $f_{0}\in B_{p,r}^\theta$ and $g\in L^1([0,T];B_{p,r}^\theta)$,
\begin{align*}
\| f(t)\| _{B_{p, r}^{\theta}} \leq\left\| f_{0}\right\| _{B_{p, r}^{\theta}}+\int_{0}^{t}\left\| g\left(t'\right)\right\| _{B_{p, r}^{\theta}} d t'+\int_{0}^{t} V'\left(t'\right)\left\| f\left(t'\right)\right\| _{B_{p, r}^{\theta}} d t'
\end{align*}
or
\begin{align*}
\| f(t)\| _{B_{p, r}^{\theta}} \leq \mathrm{e}^{C V(t)}\left(\left\| f_{0}\right\| _{B_{p, r}^{\theta}}+\int_{0}^{t} \mathrm{e}^{-C V\left(t'\right)}\left\| g\left(t'\right)\right\| _{B_{p, r}^{\theta}} d t'\right)
\end{align*}
with
\begin{align*}
V^{\prime }(t)=\begin{cases}
\| \nabla v(t)\| _{B_{p,\infty }^{\frac {d}{p}}\cap L^{\infty }},& \text{if } \theta <1+\frac {d}{p},\\
\| \nabla v(t)\| _{B_{p,r}^{\theta }},& \text{if } \theta =1+\frac {d}{p},\ r>1,\\
\| \nabla v(t)\| _{B_{p,r}^{\theta -1}},& \text{if } \theta >1+\frac {d}{p}\big(\text{or } \theta =1+\frac {d}{p},\ r=1\big).
\end{cases}
\end{align*}
If $\theta>0$, then there exists a constant $C=C(p, r, \theta)$ such that
\begin{align*}
\| f(t)\| _{B_{p, r}^{\theta}} \leq & \left\| f_{0}\right\| _{B_{p, r}^{\theta}}+\int_{0}^{t}\| g(\tau)\| _{B_{p, r}^{\theta}} d \tau+C \int_{0}^{t}\Big(\| f(\tau)\| _{B_{p, r}^{\theta}}\left\| \nabla v(\tau)\right\| _{L^{\infty}}\\
&+\left\| \nabla v(\tau)\right\| _{B_{p, r}^{\theta-1}}\left\| \nabla f(\tau)\right\| _{L^{\infty}}\Big) d \tau .
\end{align*}
\end{lemma}

\begin{lemma}\label{lem;6}
(\cite{bahouri2011fourier,MR3552560})    Let $1 \leq p<\infty$. Define $\overline{\mathbb{N}}=\mathbb{N} \cup\{\infty\}$. Suppose $f \in L^{1}(0, T ; B_{p, 1}^{\frac{d}{p}})$ and $a_{0} \in B_{p, 1}^{\frac{d}{p}}$. For $n \in \overline{\mathbb{N}}$ denote by $a^{n} \in C([0, T] ; B_{p, 1}^{\frac{d}{p}})$ the solution of
\[
\begin{cases}
\partial _{t}a^{n}+A^{n}\cdot\nabla a^{n}=f,\\
a^{n}(0,x)=a_{0}(x).
\end{cases}
\]
Assume for some $\beta \in L^{1}(0, T)$, $\underset{{n \in \overline{\mathbb{N}}}}{\sup} \left\|A^{n}\right\|_{B_{p, 1}^{1+\frac{d}{p}}} \leq\beta(t)$. If $A^{n}\to A^{\infty}$ in $L^{1}(0, T ; B_{p, 1}^{\frac{d}{p}})$, then the sequence $\{a^{n}\}_{n\in\mathbb{N}}$ converges to $a^{\infty}$ in $C([0, T] ; B_{p, 1}^{\frac{d}{p}})$.
\end{lemma}

\begin{lemma}\label{lem;7}
(\cite{bahouri2011fourier})    Let $\varepsilon \in (0,1)$. There exists a constant $C>0$ such that for any $f \in B_{\infty,\infty}^\varepsilon$,
\begin{align*}
\|f\|_{B^0_{\infty,1}} \leq \frac{C}{\varepsilon} \|f\|_{B_{\infty,\infty}^0} \left(1 + \log \frac{\|f\|_{B_{\infty,\infty}^\varepsilon}}{\|f\|_{B_{\infty,\infty}^0}}\right).
\end{align*}
\end{lemma}

\begin{lemma}\label{lem;8}
(\cite{bahouri2011fourier})     Let $\rho$ be a measurable function from $[t_0,T]$ to $[0,a]$, $\gamma$ a locally integrable function from $[t_0,T]$ to $\mathbb{R}^+$, and $\mu$ a continuous and nondecreasing function from $[0,a]$ to $\mathbb{R}^+$. Assume that, for some nonnegative real number $c$, the function $\rho$ satisfies
\begin{align*}
\rho(t) \leq c + \int_{t_0}^{t} \gamma(t')\mu(\rho(t'))\,dt' \quad \text{for a.e.}\ t\in [t_0,T]. 
\end{align*}
\begin{itemize}
\item If $c$ is positive, then we have, for a.e. $t\in [t_0,T]$,
\begin{align*}
-\mathcal{M}(\rho(t)) + \mathcal{M}(c) \leq \int_{t_0}^{t} \gamma(t')\,dt' \quad \text{with}\quad \mathcal{M}(x)=\int_{x}^{a}\frac{dr}{\mu(r)}. 
\end{align*}
\item If $c=0$ and $\mu$ is an Osgood modulus of continuity, then $\rho=0$ a.e.
\end{itemize}
\end{lemma}

\begin{lemma}\label{lem;9}
(\cite{MR3016436})  For $1\leq p\leq \infty$, let $u$ be a strong solution of (\ref{eq;scalar}). Then the $L^p$ norm of $u$ is bounded by
    \begin{align*}
        \|u\|_{L^p(\mathbb{R}^d)}\leq\|u_0\|_{L^p(\mathbb{R}^d)},\  \forall t\in[0,T). 
    \end{align*}
\end{lemma}

\begin{lemma}\label{lem;10}
(\cite{CHEN20162343})  If $f(t)\in C^1(\mathbb{R})$ satisfies
    \begin{align*}
        f'(t)\leq -af^2(t)+b,
    \end{align*}
    and $f(0)<-\sqrt{\frac{b}{a}}$, where $a,b>0$, then
    \begin{align*}
        f(t)\rightarrow-\infty, \ \text{as} \ t\rightarrow \tilde{t} \leq \frac{1}{2\sqrt{ab}}\log \left( \frac{f(0)-\sqrt{\frac{b}{a}}}{f(0)+\sqrt{\frac{b}{a}}} \right).
    \end{align*}
\end{lemma}

\begin{lemma}\label{lem;11}
(\cite{bahouri2011fourier})Let $s>0$ and $1\le p\le\infty$.
\begin{enumerate}
    \item $B_{p,1}^s(\mathbb{R}^d)\cap L^\infty(\mathbb{R}^d)$ is a Banach algebra. Moreover, $B_{\infty,1}^0(\mathbb{R}^d)\hookrightarrow L^\infty(\mathbb{R}^d)\hookrightarrow B_{\infty,\infty}^0(\mathbb{R}^d)$.
    \item
    \begin{align*}
    \lim_{j\to\infty}\|S_j u - u\|_{B_{p,1}^s}=0
    \end{align*}
    \item For any $u,v\in B_{p,1}^s(\mathbb{R}^d)\cap L^\infty(\mathbb{R}^d)$, we have
    \begin{align*}
    \|uv\|_{B_{p,1}^s}\le C\big(\|u\|_{B_{p,1}^s}\|v\|_{L^\infty}+\|v\|_{B_{p,1}^s}\|u\|_{L^\infty}\big).
    \end{align*}
    We also have the following interpolation inequality
    \begin{align*}
    \|u\|_{B_{\infty,1}^1}\le C\|u\|_{B_{\infty,\infty}^0}^{\frac12}\|u\|_{B_{\infty,\infty}^2}^{\frac12}.
    \end{align*}
    \item Let $m\in\mathbb{R}$ and $f$ be an $S^m-$multiplier (i.e., $f:\mathbb{R}^d\to\mathbb{R}^d$ is smooth and satisfies that $\forall \alpha\in\mathbb{N}^d$, there exists a constant $C_\alpha$ such that $|\partial^\alpha f(\xi)|\le C_\alpha(1+|\xi|)^{m-|\alpha|}$ for all $\xi\in\mathbb{R}^d$). Then the operator $f(D)$ is continuous from $B_{p,1}^s(\mathbb{R}^d)$ to $B_{p,1}^{s-m}(\mathbb{R}^d)$.
\end{enumerate}
\end{lemma}

\section{Local well-posedness}\label{Local well-posedness}
\begin{proof}[\normalfont\bfseries Proof of Theorem~\ref{thm;1}]
The argument for the local well-posedness theory within the framework of subcritical spaces $E^s_{p,r}$ ($s>1+\frac{d}{p}$, $1\leq p,r\leq \infty$ or $s = 1+\frac{d}{p},r=1$ with $1\leq p<\infty$) consists of three key steps outlined below.\\
\textbf{Step 1. Existence.}\\
    Set $u^0 = 0$. Define a sequence $\{u^n\}_{n\in N}$ of smooth functions by solving the following linear transport equations:
    \[
    \begin{cases}
        u^{n+1}_t + \alpha(u^n) \cdot \nabla u^{n+1} = H[u^n],\\
        u^{n+1}|_{t = 0} = u_0.
    \end{cases}
    \]
    First of all, due to Proposition \ref{prop;3}, we get
    \begin{align}\label{est}
        \|H[u]\|_{B^s_{p,r}}\leq \|u\|_{B^s_{p,r}}.
    \end{align}
        Combining Lemma\ \ref{lem;4} and (\ref{est}), we obtain a unique solution. Moreover, Lemma\ \ref{lem;5} and (\ref{est}) together give that 
    \begin{align}\label{ineq;uniform}
        \|u^{n+1}\|_{B^s_{p,r}}\leq \left(\|u_0\|_{B^s_{p,r}}+\int^t_0\mathrm{e}^{-C\int^{t'}_0\| \nabla\alpha(u^n)\|_{B^{s-1}_{p,r}}dt''}\|u^n\|_{B^s_{p,r}}dt'\right)\mathrm{e}^{C\int^{t}_0\| \nabla\alpha(u^n)\|_{B^{s-1}_{p,r}}dt'}.
    \end{align}
    Assume $\|u^n\|_{B^s_{p,r}}\leq \frac{2\|u_0\|_{B^s_{p,r}}}{1-2Ct\|u_0\|_{B^s_{p,r}}}$, $T<\min\{ \frac{1}{2C\|u_0\|_{B^s_{p,r}}},\frac{\log2}{C}\}$, combining Lemma\ \ref{lem;3} and (\ref{ineq;uniform}) we have
    \begin{align}\label{et;ini}
        \|u^{n+1}(t)\|_{B_{p,r}^s}
&\leq \|u_0\|_{B_{p,r}^s} \mathrm{e}^{C\int_0^t \|u^n\|_{B_{p,r}^s}d\tau} \times \mathrm{e}^{Ct} \\
&\leq \frac{2\|u_0\|_{B_{p,r}^s}}{1 - 2C\|u_0\|_{B_{p,r}^s} t},
\qquad \forall\, t\in [0,T].
    \end{align}
    Therefore $\{u^n\}$ is uniformly bounded in $L^{\infty}([0,T];{B^s_{p,r}})$.
    Then we will use the compactness method for approximating sequence $\{u^n\}_{n\in N}$ to get a solution $u$ of (\ref{eq;scalar}). Since $u^n$ is uniformly bounded in $L^{\infty}([0,T];{B^s_{p,r}})$ and $\|H[u]\|_{B^s_{p,r}}\leq \|u\|_{B^s_{p,r}}$, we deduce that $\partial_tu^n$ is uniformly bounded in $L^{\infty}([0,T];{B^{s-1}_{p,r}})$. Thus, $u^n$ is uniformly bounded in $C([0,T];{B^s_{p,r}}) \cap C^{\frac{1}{2}}([0,T];{B^{s-1}_{p,r}})$.

    Let $\{\phi_j\}_{j\in N}$ be a sequence of smooth functions with value in $[0,1]$ supported in the ball $B(0,j+1)$ and equal to 1 on $B(0,j)$. The map $z \rightarrow\phi_j z$ is compact from ${B^s_{p,r}}$ to ${B^{s-1}_{p,r}}$ by Theorem 2.94 in \cite{bahouri2011fourier}. Taking advantage of the Ascoli's Theorem and Cantor's diagonal process, there exists $u_j$ such that for all $\phi \in D$, $\phi_j u^n \rightarrow  u_j$ for $j \in N$. Combining the uniform boundedness of $u_n$ and Fatou's proposition for Besov spaces, we obtain a limit function $u \in L^{\infty}([0,T];{B^{s}_{p,r}})$. By virtue of interpolation inequality, $\phi u^n$ tends to $\phi u$ in $C([0,T];B^{s-\varepsilon}_{p,r})$ for any $\varepsilon>0$. It is then pass to the limit to conclude $u$ is a solution of (\ref{eq;scalar}).\\
\textbf{Step 2. Uniqueness.}\\
    Suppose that we are given 
    \begin{align*}
        (u,v)\in \left( C ([0,T];{B^{s}_{p,r}})\cap C^1([0,T];{B^{s-1}_{p,r}}) \right)^2,
    \end{align*}
    two solutions of (\ref{eq;scalar}) with initial data $u_0,v_0\in{B^{s}_{p,r}}$. Denote $w:=u-v, w_0 = u_0-v_0.$ 
    It is obvious that $w$ solves the equation 
    \begin{align*}
        \partial_tw+\alpha(u)\cdot\nabla w = \left( \alpha(v)-\alpha(u) \right)\cdot\nabla v+H[w]
    \end{align*}
    According to Lemma\ \ref{lem;2}, Lemma\ \ref{lem;1} and (\ref{est}), we obtain that
    \begin{align*}
        \|w\|_{{B^{s-1}_{p,r}}}\leq&\|w_0\|_{B^{s-1}_{p,r}}\mathrm{e}^{C\int^t_0\|\nabla \alpha(u)\|_{B^{s-1}_{p,r}}dt'}\\
        &+C\int^t_0\mathrm{e}^{-C\int^t_{t'}\|\nabla\alpha(u)\|_{B^{s-1}_{p,r}}dt''}(\|\alpha(u)-\alpha(v)\|_{B^{s-1}_{p,r}}\|v\|_{B^{s}_{p,r}}+\|w\|_{B^{s-1}_{p,r}})dt'\\
        \leq&\|w_0\|_{B^{s-1}_{p,r}}\mathrm{e}^{C\int^t_0\| \alpha(u)\|_{B^{s}_{p,r}}dt'}\\
        &+C\int^t_0\mathrm{e}^{-C\int^t_{t'}\|\alpha(u)\|_{B^{s}_{p,r}}dt''}\|w\|_{B^{s-1}_{p,r}} \left( (\|u\|_{B^{s}_{p,r}}+\|v\|_{B^{s}_{p,r}})\|v\|_{B^{s}_{p,r}}+1\right)dt'.
    \end{align*}
    Applying Gronwall's inequality completes the proof.\\
\textbf{Step 3. The continuous dependence.}\\
    Assume $u^n_0$ tends to $u^{\infty}_0$ in ${B^{s}_{p,r}}$, we find the solutions $u^n,u^{\infty}$ with a common lifespan $T$. By \textbf{step 1-step 2}, we have $u^n, u^{\infty}$ are uniformly bounded in $L^{\infty}([0,T];{B^{s}_{p,r}})$ and
    \begin{align*}
        \|(u^n-u^{\infty})(t)\|_{B^{0}_{p,\infty}}\leq C\|u^n_0 -u^{\infty}_0\|_{B^{s}_{p,r}},\forall t\in [0,T].
    \end{align*}
    Taking advantage of interpolation inequality, we see that $u^n\rightarrow u^{\infty}$ in $C([0,T];{B^{s-\varepsilon}_{p,r}})$ for any $\varepsilon>0$. Choosing $\varepsilon = 1$, we gain $u^n\rightarrow u^{\infty}$ in $C([0,T];{B^{s-1}_{p,r}})$.

    In order to prove $u^n\rightarrow u^{\infty}$ in $C([0,T];{B^{s}_{p,r}})$, we only need to prove $\nabla u^n\rightarrow \nabla u^{\infty}$ in $C([0,T];{B^{s-1}_{p,r}})$. Let $v^n_i = u^n_{x_i}$. Split $v^n_i$ into $w^n_i+z^n_i$ with $(w^n_i,z^n_i)$ satisfying
    \[
    \begin{cases}
        \partial_t w_i^n + \alpha(u^n)\cdot\nabla w_i^n = \partial_{x_i}\left(H[u^{\infty}]\right)-\partial_{x_i}\left(\alpha(u^{\infty})\right)\cdot \nabla u^\infty,\\
        w_i^n(0,x) = v^\infty_{i,0} = \partial_{x_i}u^{\infty}_0
    \end{cases}
    \]
 and
    \[
    \begin{cases}
        \partial_tz_i^n+\alpha(u^n)\cdot \nabla z_i^n= \partial_{x_i}\left( H[u^n]-H[u^{\infty}]\right)-\left( \partial_{x_i}(\alpha(u^n))\cdot \nabla u^n-\partial_{x_i}(\alpha(u^{\infty}))\cdot \nabla u^{\infty} \right),\\
        z_i^n(0,x) = v^n_{i,0}-v^{\infty}_{i,0}=\partial_{x_i}u^n_0-\partial_{x_i}u^{\infty}_0.
    \end{cases}
    \]
    On the one hand, Lemma\ \ref{lem;6} ensures $w^n \rightarrow w^{\infty}$ in $C([0,T];{B^{s-1}_{p,r}})$.
    As
    \begin{align*}
        \|\partial_{x_i}\left( H[u^n]-H[u^{\infty}]\right)\|_{B^{s-1}_{p,r}}\leq C\|u^n-u^{\infty}\|_{B^{s}_{p,r}}\leq C\left(\|u^n-u^{\infty}\|_{B^{s-1}_{p,r}}+\sum_{j=1}^d\|u^n_{x_j}-u^{\infty}_{x_j}\|_{B^{s-1}_{p,r}}\right),
    \end{align*}
    it follows from Lemma\ \ref{lem;2} that
    \begin{align*}
        \|z_i^n(t)\|_{B^{s-1}_{p,r}}\leq C\left( \|v_{i}^n(0)-v_{i}^{\infty}(0)\|_{B^{s-1}_{p,r}}+\int^t_0\|u^n-u^{\infty}\|_{B^{s-1}_{p,r}}+\sum_{j=1}^d\|w^n_j-w^{\infty}_j\|_{B^{s-1}_{p,r}}+\sum_{j=1}^d\|z^n_j\|_{B^{s-1}_{p,r}}dt'\right).
    \end{align*}
    Using the fact that $v^n_0 \rightarrow v^{\infty}_0$ in ${B^{s-1}_{p,r}}$, $u^n \rightarrow u^{\infty}$ in $C([0,T];{B^{s-1}_{p,r}})$ and $w^n \rightarrow w^{\infty}$ in $C([0,T];{B^{s-1}_{p,r}})$.
   Applying the Gronwall's inequality, we obtain that $z^n_i$ tends to $0$ in $C([0,T];{B^{s-1}_{p,r}})$. In light of Lemmas \ref{lem;4}-\ref{lem;5}, we obtain $z^{\infty}=0$ in $C([0,T];{B^{s-1}_{p,r}})$.
    Thus
    \begin{align*}
        \|v^n_i-v^{\infty}_i\|_{L^{\infty}([0,T];{B^{s-1}_{p,r}})}&\leq\|w^n_i-w^{\infty}_i\|_{L^{\infty}([0,T];{B^{s-1}_{p,r}})}+\|z^n_i-z^{\infty}_i\|_{L^{\infty}([0,T];{B^{s-1}_{p,r}})}\\
        &\leq\|w^n_i-w^{\infty}_i\|_{L^{\infty}([0,T];{B^{s-1}_{p,r}})}+\|z^n_i\|_{L^{\infty}([0,T];{B^{s-1}_{p,r}})}\rightarrow0  \ \text{as} \ n\rightarrow\infty,
    \end{align*}
    that is 
    \begin{align*}
        u^n_{x_i}\rightarrow u^{\infty}_{x_i} \ \text{in}\ C([0,T];{B^{s-1}_{p,r}}).
    \end{align*}
    We thus prove the continuous dependence of (\ref{eq;scalar}) in $C([0,T];{B^{s}_{p,r}})$ with $1\leq p < +\infty$. 
    Consequently, we finish the proof combining with \textbf{Step 1-Step 3}.

    For the endpoint case in $E^1_{\infty,1}$, the existence part follows by the same argument as before thus we only need to prove the uniqueness and the continuous dependence. Before that, let us present a useful lemma.
    \begin{lemma}
    Let $u,v \in L^{\infty} ( [0,T],B^1_{\infty,1} )$ be two solutions of (\ref{eq;scalar}) associated with $u_0$ and $v_0$. Then we have the estimates for the difference $w = u-v$
    \begin{align}
        \|w\|_{L^\infty} &\leq C \|w_0\|_{L^\infty}, \label{et;1} \\ 
        \|w\|_{B^1_{\infty,1}}&\leq C\left( \|w_0\|_{B^1_{\infty,1}}+\int_0^t\|\nabla u\|_{B^1_{\infty,1}}\|w\|_{L^\infty}d\tau \right). \label{et;2}
    \end{align}
\end{lemma}
\begin{proof}
    It is obvious that $w$ solves the equation
    \begin{align}\label{eq;difference}  
    \begin{cases}
    \partial_t w +\alpha(v)\cdot \nabla w= \left( \alpha(v)-\alpha(u) \right)\cdot \nabla u + H[w],\\
    w(0) = w_0.
    \end{cases}
    \end{align}
    Let $p\geq2$, taking the inner product of (\ref{eq;difference}) with $p|w|^{p-2}w$, then one has
    \begin{align*}
        \frac{d}{dt}\|w\|_{L^p}^p=\int_{\mathbb{R}^d}\operatorname{div}(\alpha(v))|w|^p dx+p\int_{\mathbb{R}^d}\big(\alpha(v)-\alpha(u)\big)\cdot\nabla u\,|w|^{p-2}w dx+p\int_{\mathbb{R}^d}H[w]|w|^{p-2}w dx.
    \end{align*}
    Using Young inequality, H\"{o}lder inequality and Lemma \ref{lem;3}, we have
    \begin{align*}
        \frac{d\|w\|_{L^p}^p}{dt} &\leq C\|\nabla v\|_{L^\infty}\|w\|_{L^p}^p + \|\nabla u\|_{L^\infty}\|w\|_{L^p}^p + p\int_{\mathbb{R}^d} H[w]|w|^{p-2}w\,dx
\\
        &\leq C\left( \|\nabla u\|_{L^\infty}+\|\nabla v\|_{L^\infty}+1 \right)\|w\|^p_{L^p}.
    \end{align*}
    Using the embedding $B_{\infty,1}^1(\mathbb{R}^d) \hookrightarrow \mathcal{C}^{0,1}(\mathbb{R}^d)$, and integrating the above differential inequality with respect to time $t \in [0, T]$ yields
\begin{align*}
    \|w\|_{L^p} \leq \|w_0\|_{L^p} \exp\left(C\int_0^T \big(\| u\|_{B_{\infty,1}^1} + \|v\|_{B_{\infty,1}^1}+1 \big)ds\right) \leq C\|w_0\|_{L^p}.
\end{align*}
Let $p \to \infty$, we obtain the desired estimate (\ref{et;1}).
Applying Lemma \ref{lem;3}, Lemma \ref{lem;5} and Lemma \ref{lem;11} to (\ref{eq;difference}) yields
\begin{align*}
    \|w\|_{B^1_{\infty,1}}-\|w_0\|_{B^1_{\infty,1}} \leq C\int^t_0\left( \|u\|_{B^1_{\infty,1}} +\|v\|_{B^1_{\infty,1}}+1 \right)\|w\|_{B^1_{\infty,1}}+\|\nabla u\|_{B^1_{\infty,1}}\|w\|_{L^\infty}d\tau.
\end{align*}
Gronwall's inequality yields the desired (\ref{et;2}).
\end{proof}
\noindent\textbf{Step 4. Uniqueness for the endpoint case.}\\
    Suppose 
    \begin{align*}
        (u,v)\in \left( C ([0,T];{B^{1}_{\infty,\infty}})\cap C^1([0,T];{B^{0}_{\infty,\infty}}) \right)^2,
    \end{align*}
    are two solutions of (\ref{eq;scalar}) with initial data $u_0,v_0\in{B^{1}_{\infty,\infty}}$. Denote $w:=u-v, w_0 := u_0-v_0.$ 
    Then $w$ satisfies the inequality
    \begin{align}
    \|w\|_{B^{0}_{\infty,\infty}} 
    &\leq \|w_0\|_{B^{0}_{\infty,\infty}} \mathrm{e}^{C\int^t_0 \| \alpha(u)\|_{B^{1}_{\infty,1}} dt'} \nonumber \\
    &\quad + C\int^t_0 \mathrm{e}^{-C\int^t_{t'} \|\alpha(u)\|_{B^{1}_{\infty,1}} dt''}
    \Bigl( \|(\alpha(u)-\alpha(v))\cdot\nabla v\|_{B^0_{\infty,\infty}} + \|w\|_{B^0_{\infty,\infty}} \Bigr) dt'. \label{ineq;uq}
\end{align}
    Combining Lemma\ \ref{lem;1}, Lemma\ \ref{lem;3} and Lemma\ \ref{lem;7}, we obtain the estimate
    \begin{align*}
        \|\left(\alpha(u)-\alpha(v)\right)\cdot\nabla v\|_{B^0_{\infty,\infty}} &\leq C\|\nabla v\|_{B^0_{\infty,1}}\|\left(\alpha(u)-\alpha(v)\right)\|_{B^0_{\infty,1}}\\
        & \leq C\|\nabla v\|_{B^0_{\infty,1}}(\|u\|_{B^0_{\infty,1}}+\|v\|_{B^0_{\infty,1}})\| w\|_{B^0_{\infty,\infty}}\left( 1+\log\frac{\| w\|_{B^1_{\infty,\infty}}}{\| w\|_{B^0_{\infty,\infty}}} \right).
    \end{align*}
    Substituting the estimate into (\ref{ineq;uq}), we have
    \begin{align*}
        \|w\|_{B^0_{\infty,\infty}} &\leq\|w_0\|_{B^0_{\infty,\infty}}\mathrm{e}^{C\int^t_0\| \alpha(u)\|_{B^{1}_{\infty,1}}dt'}\\
        &+C\int^t_0\mathrm{e}^{-C\int^t_{t'}\|\alpha(u)\|_{B^{1}_{\infty,1}}dt''}\|w\|_{B^0_{\infty,\infty}}\log\left(\mathrm{e}+\frac{\| w\|_{B^1_{\infty,\infty}}}{\| w\|_{B^0_{\infty,\infty}}} \right)dt'.
    \end{align*}
    Setting $u_0 = v_0$, we obtain
    \begin{align*}
        \|w\|_{B^0_{\infty,\infty}} \leq &C\int^t_0\|w\|_{B^0_{\infty,\infty}}\log\left(e+\frac{\| w\|_{B^1_{\infty,\infty}}}{\| w\|_{B^0_{\infty,\infty}}} \right)dt'\\
        \leq&C\int^t_0\|w\|_{B^0_{\infty,\infty}}\left(1-\log(\|w\|_{B^0_{\infty,\infty}})\right)\log(e+\|w\|_{B^1_{\infty,\infty}})dt'.
    \end{align*}
    Obviously, $x(1-\log x)$ is an Osgood modulus of continuity, Lemma\ \ref{lem;8} implies $w = 0 $ almost everywhere. As $B^1_{\infty,1}\hookrightarrow B^1_{\infty,\infty}$, we thus obtain the uniqueness in $B^1_{\infty,1}$.\\
\textbf{Step 5. The continuous dependence for the endpoint case.}\\
We set $\mathbf{S_t}(u_0) = $ the solution map of the radiating gas model (\ref{eq;scalar}) with initial data $u_0$. Let $u = \mathbf{S_t}(u_0)$ and $v =\mathbf{S_t}(S_Nu_0)$, using Lemma (\ref{lem;5}) and (\ref{et;1}), we have
\begin{align*}
    \|\mathbf{S_t}(S_Nu_0)\|_{B^2_{\infty,1}}\leq C\|S_Nu_0\|_{B^2_{\infty,1}}\leq C2^N\|u_0\|_{B^1_{\infty,1}}
\end{align*}
and
\begin{align*}
    \|\mathbf{S_t}(S_N u_0)-\mathbf{S_t}(u_0)\|_{L^\infty}
&\leq C\|S_N u_0 - u_0\|_{L^\infty}\\
&\leq C\|S_N u_0 - u_0\|_{B_{\infty,1}^0}\\
&\leq C 2^{-N}\|S_N u_0 - u_0\|_{B_{\infty,1}^1},
\end{align*}
which combining with (\ref{et;2}) imply
\begin{align*}
    \|\mathbf{S}_t(S_Nu_0)-\mathbf{S}_t(u_0)\|_{B_{\infty,1}^1}\leq C \|S_Nu_0-u_0\|_{B_{\infty,1}^1}.
\end{align*}
Using the triangle inequality, for $u_0$, $\widetilde u_0 \in B^1_{\infty,1}$, we have
\begin{align*}
    \|\mathbf{S_t}(u_0)-\mathbf{S_t}(\widetilde{u}_0)\|_{B_{\infty,1}^1}
\leq& \|\mathbf{S_t}(S_N u_0)-\mathbf{S_t}(u_0)\|_{B_{\infty,1}^1}\\ &+\|\mathbf{S_t}(S_N \widetilde{u}_0)-\mathbf{S_t}(\widetilde{u}_0)\|_{B_{\infty,1}^1}+\|\mathbf{S_t}(S_N u_0)-\mathbf{S_t}(S_N\widetilde{u}_0)\|_{B_{\infty,1}^1}\\
 \leq C&
  \underbrace{\|S_Nu_0-u_0\|_{B_{\infty,1}^1}}_{I} + \underbrace{\|S_N\widetilde{u}_0-\widetilde{u}_0\|_{B_{\infty,1}^1}}_{II}+\underbrace{\|\mathbf{S_t}(S_N u_0)-\mathbf{S_t}(S_N\widetilde{u}_0)\|_{B_{\infty,1}^1}}_{III}
\end{align*}
Using (\ref{et;1}) and the interpolation inequality in Lemma \ref{lem;11}, we get
\begin{align*}
    III &\leq C \bigl\|\mathbf{S}_t(S_N u_0)-\mathbf{S}_t(S_N \widetilde{u}_0)\bigr\|_{B_{\infty,\infty}^0}^{\frac12} \bigl\|\mathbf{S}_t(S_N u_0)-\mathbf{S}_t(S_N \widetilde{u}_0)\bigr\|_{B_{\infty,\infty}^2}^{\frac12}\\
&\leq C \bigl\|\mathbf{S}_t(S_N u_0)-\mathbf{S}_t(S_N \widetilde{u}_0)\bigr\|_{L^\infty}^{\frac12} \bigl\|\mathbf{S}_t(S_N u_0)-\mathbf{S}_t(S_N \widetilde{u}_0)\bigr\|_{B_{\infty,1}^2}^{\frac12}\\
&\leq C \bigl\|S_N u_0 - S_N \widetilde{u}_0\bigr\|_{L^\infty}^{\frac12} \bigl\|\mathbf{S}_t(S_N u_0)-\mathbf{S}_t(S_N \widetilde{u}_0)\bigr\|_{B_{\infty,1}^2}^{\frac12}\\
&\leq C 2^{\frac{N}{2}} \|u_0 - \widetilde{u}_0\|_{B_{\infty,1}^1}^{\frac12},
\end{align*}
which implies that
\begin{align*}
    \|\mathbf{S}_t(u_0)-\mathbf{S}_t(\widetilde{u}_0)\|_{B_{\infty,1}^1}
&\leq C \|S_N u_0 - u_0\|_{B_{\infty,1}^1}
+\|S_N \widetilde{u}_0 - \widetilde{u}_0\|_{B_{\infty,1}^1}
+2^{\frac{N}{2}}\|u_0-\widetilde{u}_0\|_{B_{\infty,1}^1}^{\frac12}.
\end{align*}
For sufficiently large $N$, we may conclude that the data-to-solution map $u_0 \rightarrow S_t(u_0)$ is continuously dependent on the initial data.
\end{proof}

\section{Nonuniform continuous dependence}\label{Nonuniform continuous dependence}
We first introduce a smooth, radial cut-off function in frequency space. Let $\hat{\phi}\in C^{\infty}_0$ be a real-valued and non-negative function satisfying
   \begin{align*}
      \widehat{\phi}(\xi)=
\begin{cases}
1, & |\xi|\leq \dfrac14,\\
0, & |\xi|\geq \dfrac12.
\end{cases}
   \end{align*}
We define the high frequency function $f_n$ and the low frequency function $g_n$ as follows:
\begin{align*}
    f_n(x) &:= 2^{-ns}\phi(x)\sin\left(\frac{17}{12}2^nx_1\right),\\
    g_n (x)&:= \frac{12}{17}2^{-n}\phi(x).
\end{align*}
\begin{lemma}\label{lem;4.1}
    Let $s\in \mathbb{R}$ and $1\leq p\leq \infty$, $1\leq r\leq\infty$. Then for any $\theta \in \mathbb{R}$, we have
    \begin{align*}
        &\|f_n\|_{L^p}\leq C2^{-n},\ \|\nabla f_n\|_{L^p}\leq C2^{-ns+n},\ \|f_n\|_{B^\theta_{p,r}}\leq C2^{n(\theta-s)},\\
        &\|g_n\|_{L^p}\leq C2^{-n},\ \|\nabla g_n\|_{L^p}\leq C2^{-n},\ \|g_n\|_{B^s_{p,r}}\leq C2^{-n-s},
    \end{align*}
    and there is a constant $M>0$ such that
    \begin{align*}
        \underset{n \rightarrow\infty}{\lim}\inf\|g_n \cdot\nabla f_n\|_{B^s_{p,\infty}} \geq M.
    \end{align*}
\end{lemma}

\begin{corollary}\label{cor;4.1}
    
    Let $(s,p)$ meet the condition (\ref{cond;un}) in Theorem \ref{thm;un}. There exists a constant $\kappa>0$ such that 
    \begin{align}
        \underset{n\rightarrow \infty}{\lim}\inf\| \alpha(f_n+g_n)\cdot \nabla(f_n+g_n)-\alpha(f_n)\cdot \nabla f_n \|_{B^s_{p,\infty}}\geq \kappa.
    \end{align}
    
\end{corollary}
\begin{proof}
Combining the triangle inequality, Lemma \ref{lem;4.1} and the embedding $B^s_{p,r}\hookrightarrow B^s_{p,\infty}$, we obtain
    \begin{align*}
\big\| \alpha(f_n+g_n)\cdot \nabla(f_n+g_n) -& \alpha(f_n)\cdot \nabla f_n \big\|_{B_{p,\infty}^s}
\\
&\ge \big\| \big(\alpha(f_n+g_n)-\alpha(f_n)\big)\cdot \nabla f_n + \alpha(f_n+g_n)\cdot \nabla g_n \big\|_{B_{p,\infty}^s} \\
&\geq C \big\| \big(\alpha(f_n+g_n)-\alpha(f_n)\big)\cdot \nabla f_n \big\|_{B_{p,\infty}^s} - \big\| \alpha(f_n+g_n)\cdot \nabla g_n \big\|_{B_{p,r}^s} \\
&\geq C \|g_n\cdot \nabla f_n\|_{B_{p,\infty}^s} - \|\alpha(f_n+g_n)\|_{B_{p,r}^s} \|\nabla g_n\|_{B_{p,r}^s} \\
&\geq C M - 2^{-n-s-1}.
\end{align*}
Taking $ n\to\infty$, there exists $\kappa$ such that
\begin{align*}
\lim_{n\to\infty}\inf\big\| \alpha(f_n+g_n)\cdot \nabla(f_n+g_n) - \alpha(f_n)\cdot \nabla f_n \big\|_{B_{p,\infty}^s} \ge \kappa.
\end{align*} 
\end{proof}
\begin{proposition}\label{prop;4.1}
    Let $(s,p,r)$ meet the condition (\ref{cond;un}) in Theorem \ref{thm;un} and $\mathbf{S}_t(u_0)$ be the solution to (\ref{eq;scalar}) with initial data $u_0$. Define $V(u_0):=H[u_0]-\alpha(u_0)\cdot\nabla u_0$. There exist $T_u>0$ and $C>0$, independent of $n$, such that for every $u_0\in \{ f_n,f_n+g_n \}$ and every $t\in[0,T_u]$,
    \begin{align*}
        \| \mathbf{S}_t(u_0)-u_0-tV(u_0)\|_{B^s_{p,r}}\leq Ct^2.
    \end{align*}
\end{proposition}
\begin{proof}
 \textbf{Case 1. }For $s>1+\frac{d}{p}, 1\leq p,r \leq \infty \ \text{or}\ s = 1+ \frac{d}{p}, r = 1 \ \text{with}\ 1\leq p < \infty,$   from Theorem $\ref{thm;1}$, we know that there exists a $T_u = T_u\big(\|u_0\|_{B_{p,r}^s}\big)$ such that (\ref{eq;scalar}) with initial data $u_0$ has a unique solution $\mathbf{S}_t(u_0)\in E_{p,r}^s(T)$ and $T\approx 1$. Moreover, we see from (\ref{et;ini}) that
\begin{align}\label{ineq;cor}
\|\mathbf{S}_t(u_0)\|_{L^\infty([0,T_u];B_{p,r}^s)} \leq C\|u_0\|_{B_{p,r}^s} \leq C. 
\end{align}
Combining Lemma \ref{lem;4.1} and (\ref{ineq;cor}), we obtain
\begin{align}
    \underset{0\leq t \leq T}{\sup}\| \mathbf{S}_t(u_0) \|_{B^{s+k}_{p,r}}\leq C2^{kn} \ (k = 0,\pm1).
\end{align}
Owing to $V(u_0):=H[u_0]-\alpha(u_0)\cdot\nabla u_0$ and the definition of $u_0$, $k = 0,\pm1$, we deduce that
\begin{align*}
    \|V(u_0)\|_{B^{s+k}_{p,r}}\leq C \|u_0\|_{B^{s+k}_{p,r}}+\|\nabla u_0\|_{L^\infty}\|u_0\|_{B^{s+k}_{p,r}}+\|u_0\|_{B^{s+k+1}_{p,r}}\|u_0\|_{B^{s+k-1}_{p,r}}\leq C 2^{kn}.
\end{align*}
Set $E = \mathbf{S}_t(u_0)-u_0-tV(u_0) $. We know that $E$ satisfies
\begin{align}\label{eq;E}
    \begin{cases}
\begin{aligned}
\partial_tE+\alpha(\mathbf{S}_t(u_0))\cdot \nabla E
&=\left( \alpha(u_0)-\alpha(\mathbf{S}_t(u_0))\right)\cdot\nabla u_0 \\
&\quad -t\alpha(\mathbf{S}_t(u_0))\cdot\nabla V(u_0)+H[E+tV(u_0)],
\end{aligned}\\
E(0,x) = 0. 
\end{cases}
\end{align}
For $k = -1$, applying Lemma \ref{lem;2}, Lemma \ref{lem;3} and Lemma \ref{lem;9} to (\ref{eq;E}), we see for all $t\in[0,T]$
\begin{align}\label{ineq-E(s-1)}
        \|E\|_{B^{s-1}_{p,r}}\leq C &\int^t_0 \|E\|_{B^{s-1}_{p,r}}+t\|V(u_0)\|_{B^{s-1}_{p,r}}+\|E+tV(u_0)\|_{L^\infty}\|u_0\|_{B^{s}_{p,r}}+\|E+tV(u_0)\|_{B^{s-1}_{p,r}}\|\nabla u_0\|_{L^\infty} \notag\\
    &+t\|\alpha(\mathbf{S}_t(u_0))\cdot\nabla V(u_0)\|_{B^{s-1}_{p,r}}+\|\nabla \alpha(\mathbf{S}_t(u_0))\|_{B^{s-2}_{p,r}}\|E\|_{B^{s-1}_{p,r}}d\tau.
\end{align}
Since $\mathbf{S}_t(u_0)$ is bounded in $L^\infty([0,T];B^s_{p,r})$, we deduce from Lemma \ref{lem;1} and Lemma \ref{lem;3} that
\begin{align*}
    &\|V(u_0)\|_{B^s_{p,r}}\leq C \|u_0\|_{B^s_{p,r}}+\|u_0\|_{B^{s-1}_{p,r}}\|u_0\|_{B^{s+1}_{p,r}} + \|u_0\|_{B^s_{p,r}}^2, \\
    &\|\alpha(\mathbf{S}_t(u_0))\cdot\nabla V(u_0)\|_{B^{s-1}_{p,r}}\leq C \|u_0\|_{L^\infty}\|\nabla V(u_0)\|_{B^{s-1}_{p,r}}+\|\nabla V(u_0)\|_{L^\infty}\|u_0\|_{B^{s-1}_{p,r}},\\
    &\|V(u_0)\|_{L^\infty}\|u_0\|_{B^{s}_{p,r}}\leq C\|\alpha(\mathbf{S}_t(u_0))\cdot\nabla u_0\|_{B^{s-1}_{p,r}}\|u_0\|_{B^{s}_{p,r}}+\|u_0\|_{B^{s-1}_{p,r}}\|u_0\|_{B^{s}_{p,r}}\leq C\|u_0\|_{B^{s-1}_{p,r}}\|u_0\|_{B^{s}_{p,r}},\\
    &\|V(u_0)\|_{B^{s-1}_{p,r}}\|\nabla u_0\|_{L^\infty}\leq C\|\alpha(\mathbf{S}_t(u_0))\cdot\nabla u_0\|_{B^{s-1}_{p,r}}\|u_0\|_{B^{s}_{p,r}}+\|u_0\|_{B^{s-1}_{p,r}}\|u_0\|_{B^{s}_{p,r}}\leq C\|u_0\|_{B^{s-1}_{p,r}}\|u_0\|_{B^{s}_{p,r}}.
\end{align*}
Combining the above inequalities, Lemma \ref{lem;4.1} and (\ref{ineq-E(s-1)}), we find
\begin{align*}
    \|E\|_{B^{s-1}_{p,r}} \leq C  t2^{-n}+\int^t_0\|E\|_{B^{s-1}_{p,r}}d\tau.
\end{align*}
Gronwall's inequality ensures 
\begin{align}\label{et;s-1}
    \|E\|_{B^{s-1}_{p,r}}\leq Ct2^{-n}.
\end{align}
For $k=0$, similar estimates imply
\begin{align}\label{ineq-E(s)}
    \|E\|_{B^s_{p,r}}\leq C &\int^t_0 \|E\|_{B^s_{p,r}}+t\|V(u_0)\|_{B^s_{p,r}}+\|E+tV(u_0)\|_{L^\infty}\|\nabla u_0\|_{B^{s}_{p,r}}+\|E+tV(u_0)\|_{B^s_{p,r}}\|\nabla u_0\|_{L^\infty}\notag \\
    &+t\|\alpha(\mathbf{S}_t(u_0))\cdot\nabla V(u_0)\|_{B^{s}_{p,r}}+\|\nabla \alpha(\mathbf{S}_t(u_0))\|_{B^{s-1}_{p,r}}\|E\|_{B^{s}_{p,r}}d\tau.
\end{align}
Similarly, according to Lemma \ref{lem;1} and Lemma \ref{lem;3}, one has
\begin{align*}
    &\|V(u_0)\|_{B^{s+1}_{p,r}}
\leq C
\|u_0\|_{B^{s+1}_{p,r}}
+\|\alpha(u_0)\|_{L^\infty}
   \|\nabla u_0\|_{B^{s+1}_{p,r}}
+\|\alpha(u_0)\|_{B^{s+1}_{p,r}}
   \|\nabla u_0\|_{L^\infty}
\leq C 2^n\\
&\|\alpha(\mathbf{S}_t(u_0))\cdot\nabla V(u_0)\|_{B^{s}_{p,r}}\leq C \|\alpha(\mathbf{S}_t(u_0))\|_{L^\infty}\|V(u_0)\|_{B^{s+1}_{p,r}}+\|\nabla V(u_0)\|_{L^\infty}\|\alpha(\mathbf{S}_t(u_0))\|_{B^{s}_{p,r}},\\
    &\|V(u_0)\|_{L^\infty}\|\nabla u_0\|_{B^{s}_{p,r}}\leq C\|\alpha(\mathbf{S}_t(u_0))\cdot\nabla u_0\|_{B^{s-1}_{p,r}}\|u_0\|_{B^{s+1}_{p,r}}+\|u_0\|_{B^{s-1}_{p,r}}\|u_0\|_{B^{s+1}_{p,r}}\\
    &\quad\quad\quad\quad\quad\quad\quad\quad\quad\ \leq C\|u_0\|_{B^{s-1}_{p,r}}\|u_0\|_{B^{s}_{p,r}}\|u_0\|_{B^{s+1}_{p,r}}+ \|u_0\|_{B^{s-1}_{p,r}}\|u_0\|_{B^{s+1}_{p,r}},\\
    &\|V(u_0)\|_{B^{s}_{p,r}}\|\nabla u_0\|_{L^\infty}\leq C\|u_0\|_{B^{s-1}_{p,r}}\|u_0\|_{B^{s}_{p,r}}\|u_0\|_{B^{s+1}_{p,r}}+\|u_0\|_{B^{s}_{p,r}}^3.
\end{align*}
Plugging (\ref{et;s-1}) and above inequality into (\ref{ineq-E(s)}) and using Lemma \ref{lem;4.1}, we infer
\begin{align*}
    \|E\|_{B^s_{p,r}}\leq C t^2+\int^t_0\|E\|_{B^{s}_{p,r}}d\tau. 
\end{align*}
Using Gronwall's inequality completes the proof.\\
\textbf{Case 2.} For the endpoint Besov space $B^1_{\infty,1}$, we  assume that $\|u_0\|_{B^1_{\infty,1}}\leq C$.
By the Newton-Leibniz formula, Lemma \ref{lem;1}, Lemma \ref{lem;3}, Lemma \ref{lem;9} and Lemma \ref{lem;4.1}, one has
\begin{align}
    \|\mathbf{S}_t(u_0)-u_0\|_{L^{\infty}}&\leq \int^t_0\|\partial_{\tau}\mathbf{S}_t(u_0)\|_{L^{\infty}}d\tau\notag\\&\leq \int^t_0\|\alpha(\mathbf{S}_t(u_0))\cdot\nabla \mathbf{S}_t(u_0)\|_{L^{\infty}}+\|H[\mathbf{S}_t(u_0)]\|_{L^{\infty}}d\tau\notag\\
    &\leq Ct\|u_0\|_{L^{\infty}}, \label{et-infty}
\end{align}
\begin{align}
    \|u-u_0\|_{B^1_{\infty,1}}&\leq \int^t_0\|\partial_{\tau}\mathbf{S}_t(u_0)\|_{B^1_{\infty,1}}d\tau\notag\\
    &\leq\int^t_0\|\alpha(\mathbf{S}_t(u_0))\cdot\nabla \mathbf{S}_t(u_0)\|_{B^1_{\infty,1}}+\|H[\mathbf{S}_t(u_0)]\|_{B^1_{\infty,1}}d\tau\notag\\
    &\leq Ct\left( \|\mathbf{S}_t(u_0)\|_{L^{\infty}}\|\nabla \mathbf{S}_t(u_0)\|_{B^1_{\infty,1}}+\|\mathbf{S}_t(u_0)\|_{B^1_{\infty,1}}\|\nabla \mathbf{S}_t(u_0)\|_{L^\infty}+\|\mathbf{S}_t(u_0)\|_{B^1_{\infty,1}} \right)\notag\\
    &\leq Ct\left( \|u_0\|_{L^\infty}\|u_0\|_{B^2_{\infty,1}} +\|u_0\|^2_{B^1_{\infty,1}}+\|u_0\|_{B^1_{\infty,1}}\right)\notag\\
    & \leq Ct  \label{et-B1}
\end{align}

and
\begin{align}
    \|u-u_0\|_{B^2_{\infty,1}}&\leq \int^t_0\|\partial_{\tau}\mathbf{S}_t(u_0)\|_{B^2_{\infty,1}}d\tau\notag\\
    &\leq\int^t_0\|\alpha(\mathbf{S}_t(u_0))\cdot\nabla \mathbf{S}_t(u_0)\|_{B^2_{\infty,1}}+\|H[\mathbf{S}_t(u_0)]\|_{B^2_{\infty,1}}d\tau\notag\\
    &\leq Ct\left( \|\mathbf{S}_t(u_0)\|_{L^{\infty}}\|\nabla \mathbf{S}_t(u_0)\|_{B^2_{\infty,1}}+\|\mathbf{S}_t(u_0)\|_{B^2_{\infty,1}}\|\nabla \mathbf{S}_t(u_0)\|_{L^\infty}+\|\mathbf{S}_t(u_0)\|_{B^1_{\infty,1}} \right)\notag\\
    &\leq Ct\left( \|u_0\|_{L^\infty}\|u_0\|_{B^3_{\infty,1}} +\|u_0\|_{B^2_{\infty,1}}\|u_0\|_{B^1_{\infty,1}}+\|u_0\|_{B^2_{\infty,1}}\right)\notag\\
    & \leq C2^nt.\label{et-B2}
\end{align}
By the Newton-Leibniz formula once again, we obtain
\begin{align}
    \|\mathbf{S}_t(u_0)-&u_0-tV(u_0)\|_{B^1_{\infty,1}}\notag\\
    \leq &\int^t_0\|\partial_t\mathbf{S}_t(u_0)-V(u_0)\|_{B^1_{\infty,1}}d\tau\notag\\
    \leq &\int^t_0 \|H[\mathbf{S}_t(u_0)-u_0]+\alpha(u_0)\cdot(\nabla u_0-\nabla \mathbf{S}_t(u_0))+(\alpha(u_0)-\alpha(\mathbf{S}_t(u_0)))\cdot \nabla u\|_{B^1_{\infty,1}}d\tau\notag\\
    \leq &\int^t_0 \|\mathbf{S}_t(u_0)-u_0\|_{B^1_{\infty,1}} +\|u_0\|_{L^\infty}\|\mathbf{S}_t(u_0)-u_0\|_{B^2_{\infty,1}}+\|\mathbf{S}_t(u_0)-u_0\|_{B^1_{\infty,1}}\|u_0\|_{B^1_{\infty,1}}\notag\\
    &+\|\mathbf{S}_t(u_0)-u_0\|_{L^\infty}\|u\|_{B^2_{\infty,1}}+\|\mathbf{S}_t(u_0)-u_0\|_{B^1_{\infty,1}}\|\mathbf{S}_t(u_0)\|_{B^1_{\infty,1}}d\tau.\label{et-V}
 \end{align}
 Applying (\ref{et-infty}), (\ref{et-B1}) and (\ref{et-B2}) to (\ref{et-V}), Lemma \ref{lem;4.1} ensures
 \begin{align*}
     \|\mathbf{S}_t(u_0)-u_0-tV(u_0)\|_{B^1_{\infty,1}}\leq Ct^2.
 \end{align*}
Combining above two cases, we finish the proof of Proposition \ref{prop;4.1}.
\end{proof}
 \begin{proof}[\normalfont\bfseries Proof of Theorem~\ref{thm;un}]
Let $V(u_0)$ be as in (\ref{prop;4.1}) and set
\begin{align*}
    E_1 := \mathbf{S}_t(f_n + g_n) - (f_n + g_n) - tV(f_n + g_n),\quad E_2 := \mathbf{S}_t(f_n) - f_n - tV(f_n).
\end{align*}
Subtracting the two expansions gives
\begin{align}
    \mathbf{S}_t(f_n + g_n)-\mathbf{S}_t(f_n)&=g_n+t(V(f_n+g_n)-V(f_n))+E_1-E_2\notag \\
    &=g_n-t(\alpha(f_n+g_n)\cdot \nabla(f_n+g_n)-\alpha(f_n)\cdot \nabla f_n)+tH[g_n]+E_1-E_2 \label{eq;triangle}
\end{align}
Taking the $B^s_{p,r}$ norm in (\ref{eq;triangle}), combining Lemma \ref{lem;4.1}, Corollary \ref{cor;4.1} and Proposition \ref{prop;4.1}, we get for $0\leq t\leq T$,
\begin{align*}
    \|\mathbf{S}_t(f_n + g_n)-\mathbf{S}_t(f_n)\|_{B^s_{p,r}}&\geq t\| \alpha(f_n+g_n)\cdot \nabla(f_n+g_n)-\alpha(f_n)\cdot \nabla f_n \|_{B^s_{p,\infty}}-C(1+t)2^{-n} -2Ct^2\\
    &\geq\kappa t - C(1+t)2^{-n}-2Ct^2.
\end{align*}
We finally obtain
\begin{align*}
    \lim_{n\to\infty}\inf\|\mathbf{S}_t(f_n + g_n)-\mathbf{S}_t(f_n)\|_{B_{p,r}^s} \geq C t,\quad \text{for } t>0 \text{ small enough}.
\end{align*}
This proves Theorem \ref{thm;un}.
\end{proof}

\section{Blow-up}\label{Blow-up}
We rewrite the equation in (\ref{eq;scalar}) as:
\begin{align*}
    u_t+\nabla\cdot f(u) = -u +(I-\Delta)^{-1}u,
\end{align*}
where $(I-\Delta)^{-1}u = K*u$. The fundamental solution to the operator $I-\Delta$ in $\mathbb{R}^d$ is
\begin{align*}
    K(x) = (4\pi)^{-\frac{d}{2}}\int^{\infty}_0\mathrm{e}^{-s-\frac{|x|^2}{4s}}s^{-\frac{d}{2}}ds.
\end{align*}
Before investigating the conditions for the blow-up of strong solutions to the Cauchy problem (\ref{eq;scalar}), we first establish the precise blow-up scenario for the system in multi-dimensional Besov space.

\subsection{Blow-up scenario}
\begin{proposition}
    \label{thm;2}
    Let $u_0 \in B^s_{p,r}$ as in Theorem \ref{thm;1}, and let $T^*$ be the maximal existence time of the solution $u$ to Cauchy problem (\ref{eq;scalar}), then $u(t,x)$ blows up in finite time $T^*<\infty$ if and only if
\begin{align*}
\int_{0}^{T^*} \|\nabla u\|_{B^0_{\infty,\infty}}dt = +\infty.
\end{align*}
\end{proposition}
\begin{proof}
    
   Applying Lemma \ref{lem;2} to (\ref{eq;scalar}) and using the fact that $\|H[u]\|_{B^s_{p,r}} \leq \|u\|_{B^s_{p,r}}$ yields
   \begin{align*}
       \|u \|_{B^s_{p,r}}\mathrm{e}^{-C\int^{t}_0\| \nabla\alpha(u)\|_{L^\infty}dt'}\leq  \|u_0\|_{B^s_{p,r}}+\int^t_0\mathrm{e}^{-C\int^{t'}_0\| \nabla\alpha(u)\|_{{L^\infty}}dt''}\|u\|_{B^s_{p,r}}dt' .
   \end{align*}
   Combining Lemma \ref{lem;3} and Gronwall's inequality, we obtain
   \begin{align}\label{ineq;thm2}
       \|u \|_{B^s_{p,r}} &\leq \|u_0\|_{B^s_{p,r}}\mathrm{e}^{C\int^{t}_0\| \nabla\alpha(u)\|_{{L^\infty}}dt'}\notag\\ 
       &\leq \|u_0\|_{B^s_{p,r}}\mathrm{e}^{C\int^{t}_0\| \nabla u\|_{{L^\infty}}dt'}.
   \end{align}
  Moreover, we have the logarithmic interpolation inequality
  \begin{align}\label{ineq;intp}
      \|\nabla u\|_{{L^\infty}} \leq C\left(1+\|\nabla u\|_{B_{\infty,\infty}^0}\log\left(\mathrm{e}+\|u\|_{B_{p,r}^s}\right)\right)
  \end{align}
  Plugging (\ref{ineq;intp}) into (\ref{ineq;thm2}),
   we obtain
   \begin{align*}
       \|u(t)\|_{B_{p,r}^s} \leq \|u_0\|_{B_{p,r}^s} \mathrm{e}^{Ct}\exp\left(C \int_0^t \|\nabla u\|_{B_{\infty,\infty}^0}\log\left(\mathrm{e}+\|u\|_{B_{p,r}^s}\right) dt'\right).
   \end{align*}
   Therefore, easy calculations lead to
   \begin{align*}
       \log\big(\mathrm{e}+\|u(t)\|_{B_{p,r}^s}\big) \leq \log\big(\mathrm{e}+\|u_0\|_{B_{p,r}^s}\big)+Ct+C\int_0^t \|\nabla u\|_{B_{\infty,\infty}^0}\log\big(\mathrm{e}+\|u\|_{B_{p,r}^s}\big)dt'.
 \end{align*}
The Gronwall's inequality thus yields
   \begin{align}\label{ineq-thm2}
       \log\big(\mathrm{e} + \|u(t)\|_{B_{p,r}^s}\big) \leq \left(\log\big(\mathrm{e} + \|u_0\|_{B_{p,r}^s}\big)+Ct\right) \mathrm{e}^{C\int_0^t\|\nabla u\|_{B_{\infty,\infty}^0}dt'}.
   \end{align}
   Assume that $T^*$ is the maximal time of the solution to the Cauchy problem to (\ref{eq;scalar}). If 
   \begin{align*}
       T^*<\infty \ \text{and}\ \int_{0}^{T^*} \| \nabla u\|_{B^0_{\infty,\infty}}dt < +\infty ,
   \end{align*}
   then from (\ref{ineq-thm2}) we know that $\underset{{0\leq t<T^*}}{\sup}\|u(t)\|_{B^s_{p,r}}<\infty$ which contradicts the fact that $T^*$ is the maximal time of existence to the Cauchy problem (\ref{eq;scalar}). We thus have completed the proof of Proposition \ref{thm;2}.
       \end{proof}
\subsection{Blow-up conditions}    

To establish the first blow-up criterion, we prove the following maximum principle for solutions to the Cauchy problem (\ref{eq;scalar}).
  \begin{lemma}
      Suppose $u(x,t)$ is a classical solution of the Cauchy problem (\ref{eq;scalar}) on $\mathbb{R}^d\times[0,T]$ with an arbitrary $T>0$. Then $u(x,t)$ and $\underset{i=1}{\overset{d}{\sum}}
\partial_{x_i}u(x,t)$ satisfy 
  \begin{align}
  &\underset{x\in \mathbb{R}^d}{\inf}u_0(x)\leq u(x,t)\leq\underset{x\in \mathbb{R}^d}{\sup}u_0(x) ,\label{ineq;1} \\
  &\sum_{i=1}^{d}\partial_{x_i}u \leq  b_0,\label{ineq;2} \\
  &-\frac{d\delta}{2}\leq\left( K*\sum_{i=1}^{d}\partial_{x_i}u\right)(x,t)\leq k.\label{ineq;3}
  \end{align}
  \end{lemma}
  \begin{proof}
  Define the characteristics associated to the system (\ref{eq;scalar}) as
    \begin{align}\label{characteristic;curve}
        \begin{cases}
            \frac{d X_i}{d t} = f'_i(t,u(t,X)),\\
            X(0,x) = x.
        \end{cases}
    \end{align}
Along the characteristic curve (\ref{characteristic;curve}) we have
  \begin{align}\label{charactersitic;eq}
      \frac{d}{d\tau}u\left( X(\tau),\tau \right)+u(X(\tau),\tau) = (K*u)\left(X(\tau),\tau\right).
  \end{align}
   Applying $\partial_{x_j}$ on both sides of the equation, we have
    \begin{align*}
        \partial_{x_j}u_t + \sum_{i=1}^{d}f^{'}_i(u)\partial_{x_ix_j}u = -\partial_{x_j}u + K*\partial_{x_j}u - \sum_{i=1}^{d}f^{''}_i(u)\partial_{x_i}u\partial_{x_j}u.
    \end{align*}
    We set $\gamma(t) =\overset{d}{\underset{i=1}{\sum}}
u_{x_i}(t,X(t,x_0))$, then
\begin{align}\label{characteristic;ineq}
        \gamma' \leq& -\gamma+\sum_{j=1}^{d}K*\partial_{x_j}u - f''_i(u)(\sum_{i=1}^{d}\partial_{x_i}u)^2 \notag\\
        \leq &-\gamma-\alpha_0\gamma^2+\sum_{j=1}^{d}K*\partial_{x_j}u.
    \end{align}
  Then integrating (\ref{characteristic;ineq}) and (\ref{charactersitic;eq}) with respect to $\tau\in[0,t]$, we obtain
 \begin{align}
     \gamma(t)&\leq \mathrm{e}^{-t} \gamma(0) 
     +\int^t_0\mathrm{e}^{-(t-\tau)}\left( K* \gamma - \alpha_0 \gamma^2 \right)(\tau)d\tau ,\label{chara;ineq}\\
     u\left(X(t),t\right) &= \mathrm{e}^{-t}u_0(x) + \int_0^t\mathrm{e}^{-(t-\tau)}(K*u)(X(\tau),\tau)d\tau.\label{chara;eq}
 \end{align}
 Let $A(t) = \underset{x\in \mathbb{R}^d}{\inf}u(x,t),\ B(t) = \underset{x\in \mathbb{R}^d}{\sup}u(x,t)$ , from (\ref{chara;eq}) and the  fact that $A(t)\leq K*u(t) \leq B(t)$ on the basis of the definition of $K(x)$, we have
 \begin{align*}
     A(t) \geq \mathrm{e}^{-t}A(0)+\int^t_0\mathrm{e}^{-(t-\tau)}A(\tau)d\tau.
 \end{align*}
 Using Gronwall's inequality, we have $u(x,t)\geq A(t)\geq A(0) = \underset{x\in \mathbb{R}^d}{\inf}u_0(x)$. Through analogous discussions, we can deduce $u(x,t)\leq B(t)\leq\underset{x\in \mathbb{R}^d}{\sup}u_0(x) $.
 Let $B_1(t) = \underset{x\in \mathbb{R}^d}{\sup}\underset{i=1}{\overset{d}{\sum}}
\partial_{x_i}u(x,t)$, combining the inequality (\ref{chara;ineq}) we have that
 \begin{align*}
     B_1(t)\leq \mathrm{e}^{-t}b_0 +\int^t_0\mathrm{e}^{-(t-\tau)}B_1(\tau)d\tau.
 \end{align*}
 Thus Gronwall's inequality ensures (\ref{ineq;2}) and $\left( K*\overset{d}{\underset{i=1}{\sum}}
\partial_{x_i}u\right)(x,t)\leq b_0$.
 Furthermore, we can write
 \begin{align}\label{eq;invo}
     \partial_{x_i}\left( K*u \right) = \int_{\mathbb{R}^{d-1}}\int_0^{\infty}\partial_{x_i}K(y,w)[u(x-ye_i-w)-u(x+ye_i-w)]dydw,
 \end{align}
 where $e_i$ is i-th standard unit basis vector and $w = (w_1,\dots,w_{i-1},w_{i+1},\dots,w_d)$.

Combining (\ref{ineq;1}) and (\ref{eq;invo}) we have estimate
\begin{align*}
    \| \partial_{x_i}\left( K*u \right) \|_{L^{\infty}}\leq \delta \int_{\mathbb{R}^{d-1}}\int_0^{\infty}|\partial_{x_i}K(y,w)|dydw \leq\frac{\delta}{2},
\end{align*}
 which proves the left-side inequality of (\ref{ineq;3}). 
  \end{proof}
\begin{proof}[\normalfont\bfseries Proof of Theorem~\ref{thm;blowup}]
Following the above maximum principle, we proceed to the discussion of blow-up conditions. Along characteristic curves analogous to those in (\ref{characteristic;curve}), we recover the same estimate as (\ref{chara;ineq}). Substituting the inequality (\ref{ineq;3}) into (\ref{characteristic;ineq}), we obtain
\begin{align}\label{eq;ode}
    \gamma' \leq -\alpha_0\gamma^2-\gamma+ k.
\end{align}
Since $k \geq 0$, the equation $\alpha_0\gamma^2+\gamma-k = 0$ has two real roots $\gamma_1$ and $\gamma_2$:
\begin{align*}
    \gamma_1 = \frac{-1-\sqrt{1+4\alpha_0k}}{2\alpha_0} \leq-\frac{1}{\alpha_0} \ \text{and} \ \gamma_2 = \frac{-1+\sqrt{1+4\alpha_0k}}{2\alpha_0} \geq 0.
\end{align*}
Thus the differential inequality (\ref{eq;ode}) is rewritten as
\begin{align*}
    \gamma' \leq -\alpha_0(\gamma-\gamma_1)(\gamma-\gamma_2) \ \text{with} \ \gamma(0) = \sum_{i=1}^{d}\partial_{x_i} u_0(x_0).
\end{align*}
    Since $\gamma(0)<\gamma_1$, then $\gamma(t)\leq \gamma(0)<\gamma_1$ for all $t\in[0,T)$.
By a direct calculation, we obtain
\begin{align*}
    \frac{1}{\gamma_1-\gamma
    _2}\left( \frac{1}{\gamma(t)-\gamma_1}-\frac{1}{\gamma(t)-\gamma_2} \right)\gamma'(t)\leq -\alpha_0,
\end{align*}
thus we have
\begin{align*}
    \gamma(t)\leq \frac{-\gamma_1+\gamma_2\frac{|\gamma(0)-\gamma_1|}{|\gamma(0)-\gamma_2|}\cdot \mathrm{e}^{-\alpha_0(\gamma_1-\gamma_2)t}}{\frac{|\gamma(0)-\gamma_1|}{|\gamma(0)-\gamma_2|}\cdot \mathrm{e}^{-\alpha_0(\gamma_1-\gamma_2)t}-1}.
\end{align*}
Define $R:=\frac{-1+\sqrt{1+4\alpha_0k}-2\alpha_0\underset{i=1}{\overset{d}{\sum}}\partial_{x_i}u_0(x_0)}{-1-\sqrt{1+4\alpha_0k}-2\alpha_0\underset{i=1}{\overset{d}{\sum}}\partial_{x_i}u_0(x_0)},\quad \text{notice that } R>1
$, we find out that $\gamma(t)$ blows up to $-\infty$ before the time $t$ reaches $T\leq\frac{\log R}{\sqrt{1+4\alpha_0k}} .$
\end{proof}

Before presenting the second blow-up criterion, we state the following lemma, which is a higher-dimensional version of Theorem 2.1 in \cite{constantin1998wave}.
\begin{lemma}\label{lem:original-lemma-2-1}
Let $T>0$ and $v\in C^1([0,T);H^s(\mathbb R^d))$ with
$s>1+\frac d2$. Then for every $t\in[0,T)$, there exists at
least one point $\xi(t)\in\mathbb R^d$ such that
\begin{align*}
m(t)=\inf_{x\in\mathbb R^d}
\left\{\sum_{i=1}^{d}v_{x_i}(t,x)\right\}
=\sum_{i=1}^{d}v_{x_i}(t,\xi(t)).
\end{align*}
The function $m(t)$ is almost everywhere differentiable on
$[0,T)$ with
\begin{align*}
\frac{dm}{dt}
=\sum_{i=1}^{d}v_{tx_i}(t,\xi(t)).
\end{align*}
\end{lemma}

\begin{proof}

Fix $t\in[0,T)$ and define
\begin{align*}
m(t):=\inf_{x\in\mathbb R^d}
\left\{\sum_{i=1}^{d}v_{x_i}(t,x)\right\},
\qquad
m_i(t):=\inf_{x\in\mathbb R^d}\{v_{x_i}(t,x)\}.
\end{align*}

If $v(t,\cdot)\equiv0$, then $m(t)=0$, and any point may
be chosen as $\xi(t)$. We may therefore assume that
$v(t,\cdot)\not\equiv0$.

If $m_i(t)\ge0$, we have that $v(t,\cdot)$ is non-decreasing
in the $x_i$ direction. Recall that
$v(t,\cdot)\in L^2(\mathbb R^d)$, which implies
$v(t,\cdot)\equiv0$. So we may assume $m_i(t)<0$.

We further show that $m(t)<0$. Otherwise,
\begin{align*}
\sum_{i=1}^{d}v_{x_i}(t,x)\ge0
\qquad\text{for all }x\in\mathbb R^d.
\end{align*}
For each fixed $x\in\mathbb R^d$, this gives
\begin{align*}
\frac{d}{dr}v(t,x+r(1,\ldots,1))
=\sum_{i=1}^{d}v_{x_i}(t,x+r(1,\ldots,1))
\ge0.
\end{align*}
Thus $r\mapsto v(t,x+r(1,\ldots,1))$ is non-decreasing.
By Sobolev embedding, $v(t,\cdot)$ is continuous and vanishes
at infinity, so
\begin{align*}
\lim_{r\to-\infty}v(t,x+r(1,\ldots,1))
=\lim_{r\to+\infty}v(t,x+r(1,\ldots,1))
=0.
\end{align*}
Consequently, this function vanishes identically.
Since $x$ is arbitrary, we obtain $v(t,\cdot)\equiv0$,
a contradiction. Hence $m(t)<0$.

Since $v_{x_i}(t,\cdot)\in H^{s-1}(\mathbb R^d)$, we have
$\underset{|x|\to\infty}{\lim}v_{x_i}(t,x)=0$.
Then there exists at least one point $\xi(t)\in\mathbb R^d$
such that
\begin{align*}
m(t)=\sum_{i=1}^{d}v_{x_i}(t,\xi(t)).
\end{align*}

We assume that $t,t'\in[0,T)$ are fixed.
If $m(t)<m(t')$, then
\begin{align*}
\begin{aligned}
0\le m(t')-m(t)
&\le\inf_{x\in\mathbb R^d}
\left(
\sum_{i=1}^{d}v_{x_i}(t',x)
-\sum_{i=1}^{d}v_{x_i}(t,\xi(t))
\right)\\
&\le\sum_{i=1}^{d}
\left[v_{x_i}(t',\xi(t))-v_{x_i}(t,\xi(t))\right].
\end{aligned}
\end{align*}

By Sobolev embedding,
\begin{align*}
\begin{aligned}
|m(t)-m(t')|
&\leq C
\left\|\sum_{i=1}^{d}
\bigl(v_{x_i}(t)-v_{x_i}(t')\bigr)\right\|_{H^{s-1}}\\
&\leq C
|t-t'|\max_{0\le\tau\le\max\{t',t\}}
\left\|\sum_{i=1}^{d}v_{tx_i}(\tau)\right\|_{H^{s-1}}.
\end{aligned}
\end{align*}
Since $v_{tx_i}\in C([0,T);H^{s-1}(\mathbb R^d))$, we see
that $m$ is locally Lipschitz continuous on $[0,T)$.
Therefore, Rademacher's theorem implies that $m$ is differentiable almost everywhere on $(0,T)$.

Fix $t\in(0,T)$. We have
\begin{align*}
\lim_{h\to0}\sup_{y\in\mathbb R^d}
\left(
\frac{v_{x_i}(t+h,y)-v_{x_i}(t,y)}{h}
-v_{tx_i}(t,y)
\right)=0.
\end{align*}
By the definition of $m$,
\begin{align*}
m(t+h)
=\sum_{i=1}^{d}v_{x_i}(t+h,\xi(t+h))
\le\sum_{i=1}^{d}v_{x_i}(t+h,\xi(t)).
\end{align*}
Consequently, setting $h>0$, we obtain
\begin{align*}
\lim_{h\to0^+}\sup
\frac{m(t+h)-m(t)}{h}
\le\sum_{i=1}^{d}v_{tx_i}(t,\xi(t)).
\end{align*}
On the other hand,
\begin{align*}
\lim_{h\to0^+}\inf
\frac{m(t)-m(t-h)}{h}
\ge\sum_{i=1}^{d}v_{tx_i}(t,\xi(t)).
\end{align*}
Since $m$ is differentiable almost everywhere,
\begin{align*}
\frac{dm}{dt}
=\sum_{i=1}^{d}v_{tx_i}(t,\xi(t)),
\end{align*}
which completes the proof.
\end{proof}
\begin{proof}[\normalfont\bfseries Proof of Theorem~\ref{thm;Mm}]
Applying $\partial_{x_j}$ on both sides of the equation \ref{eq;scalar}, we have
\begin{align*}
\partial_{x_j}u_t+\sum_{i=1}^d f'_i(u)\partial_{x_i x_j}u = -\partial_{x_j}u + K*\partial_{x_j}u -\sum_{i=1}^d f''_i(u)\partial_{x_i}u\partial_{x_j}u.
\end{align*}
By Lemma \ref{lem:original-lemma-2-1}, we obtain
\begin{align*}
m'(t)+m(t)\sum_{i=1}^d f''_i\partial_i u &= -m(t)+\sum_{i=1}^d \big[K*u_{x_i}\big]
\end{align*}
\noindent and
\begin{align*}
M'(t)+M(t)\sum_{i=1}^d f''_i\partial_i u &= -M(t)+\sum_{i=1}^d \big[K*u_{x_i}\big].
\end{align*}
 By observing
\begin{align*}
M(t)\ge \sum_{i=1}^d \big[K*u_{x_i}\big]\ge m(t),
\end{align*}
 we can obtain
\begin{align*}
m'(t)&\le -m(t)\sum_{i=1}^d f''_i\partial_i u + M(t)-m(t)\\
&\le -\alpha_0 m(t)^2 + M(t)-m(t).
\end{align*}
For a.e. $t\in[0,T)$, we have
\begin{align}\label{ineq;Mm}
    M'(t) \le -\alpha_0 M(t)^2,\quad m'(t) \le -\alpha_0 m(t)^2 - m(t) + M(t).
\end{align}
Integrating the first inequality in (\ref{ineq;Mm}), we obtain
\begin{align}\label{zyM}
M(t)\le\frac{M(0)}{1+\alpha_0M(0)t},
\end{align}
thus
\begin{equation}\label{eq:improved-lower-slope}
m'(t)\le-\alpha_0m(t)^2-m(t)
+\frac{M(0)}{1+\alpha_0M(0)t}.
\end{equation}
Define
\begin{equation}
a(t):=
\frac{3}{4}+\frac{1}{4}\sqrt{1+\dfrac{8\alpha_0M(0)}
{1+\alpha_0M(0)t}}\quad \text{and}\quad 
\psi(t):=-\alpha_0m(t)-a(t).
\end{equation}
Then $a(t)\ge1$, and direct calculation gives
\begin{equation}\label{eq:improved-barrier}
(a-1)(2a-1)
=\frac{\alpha_0M(0)}{1+\alpha_0M(0)t}
\quad \text{and}\quad
a'=-\frac{(a-1)^2(2a-1)^2}{4a-3}.
\end{equation}
Combining \eqref{eq:improved-lower-slope} and
\eqref{eq:improved-barrier}, we infer that
\begin{align*}
\psi'
&=-\alpha_0m'-a'\\
&\ge(\psi+a)^2-(\psi+a)-(a-1)(2a-1)-a'\\
&=\psi^2+(2a-1)\psi+\frac{4(a-1)^4}{4a-3}.
\end{align*}
Since $a\ge1$, a continuity argument yields $\psi(t)>0$ on $[0,T)$.
Thus $
\psi'\ge \psi^2+\psi
\ \text{ for a.e. }t\in[0,T).
$
Thanks to the condition (\ref{cond;bu}) and Lemma \ref{lem;10}, $m\rightarrow-\infty$ as $t\rightarrow T$, where $T$ is estimated as
\begin{align*} T\le\log\left(1+\frac1{\psi(0)}\right)=\log \left(
1-\frac{4}{4\alpha_0 m(0)+3+\sqrt{1+8\alpha_0 M(0)}}\right).
\end{align*}
\end{proof}

\section{Ill-posedness}\label{Ill-posedness}
\begin{proof}[\normalfont\bfseries Proof of Theorem~\ref{thm;6}]
Fix $\sigma\in(0,1)$ with $\|(n^{-\sigma})_{n\geq2}\|_{l^{r}}<\infty$ for
$\frac1r<\sigma<1$. Let $\varphi$ denote the function occurring in the definition of the dyadic blocks in Section 2 and choose
$\tilde\varphi\in\mathcal{D}(\mathbb{R}^d)$, radial, real, non-negative and
non-zero, with 
\begin{align*}
    \operatorname{supp}\tilde\varphi\subset\Big\{1.4\leq|\xi|\leq1.49\Big\}.
\end{align*}
Put $e:=d^{-1/2}(1,\dots,1)$ and, for $n\geq2$, define $h_n$ by
\begin{align*}
    \widehat{h_n}(\xi):=i\,2^{-n}\,(\xi\cdot e)\,
    \tilde\varphi(2^{-n}\xi),
    \quad\text{that is}\quad
    h_n(x)=2^{nd}\big(\partial_{e}\check{\tilde\varphi}\big)(2^nx)
\end{align*}
and set
\begin{align*}
    u_0^{(N)}:=(\log N)^{-1}\sum_{n=2}^{N}n^{-\sigma}2^{-n(d+1)}h_n .
\end{align*}
By direct computation, we obtain
\begin{align*}
    \|h_n\|_{L^{p}}&=\|\partial_{e}\check{\tilde\varphi}\|_{L^{p}}\,2^{nd(1-\frac{1}{p})},\
    \|h_n\|_{L^{\infty}}=\|\partial_{e}\check{\tilde\varphi}\|_{L^{\infty}}\,2^{nd}, \
    \partial_{e}h_n(0)=-{\int_{\mathbb{R}^d}\tilde\varphi(\xi)\,(\xi\cdot e)^2
    \,\mathrm d\xi}\,2^{n(d+1)}.
\end{align*}
Hence
\begin{align*}
    2^{\,n(1+\frac{d}{p})}\big\|\Delta_nu_0^{(N)}\big\|_{L^{p}}
    =\|\partial_{e}\check{\tilde\varphi}\|_{L^{p}}\,(\log N)^{-1}n^{-\sigma},
\end{align*}
which implies 
\begin{align*}
    \big\|u_0^{(N)}\big\|_{B^{1+\frac{d}{p}}_{p,r}}\leq \|\partial_{e}\check{\tilde\varphi}\|_{L^{p}}\|(n^{-\sigma})_{n\geq2}\|_{l^{r}}(\log N)^{-1},
    \
    \big\|u_0^{(N)}\big\|_{L^{\infty}}\leq \|\partial_{e}\check{\tilde\varphi}\|_{L^{\infty}}(\log N)^{-1},
\end{align*}
while
\begin{align*}
    \gamma(0)&:=\sum_{i=1}^{d}\partial_{x_i}u_0^{(N)}(0)
    \leq-\frac{\sqrt d\int_{\mathbb{R}^d}\tilde\varphi(\xi)\,(\xi\cdot e)^2
    \mathrm d\xi}{1-\sigma}\,(\log N)^{-1}
    \big((N+1)^{1-\sigma}-2^{1-\sigma}\big).
\end{align*}
For $N$ large enough, we have 
\begin{align*}
\sqrt{1+4 d\alpha_0\|u_0^{(N)}\|_{L^{\infty}}}\leq\sqrt2\  \text{and} \ -2\alpha_0\gamma(0)-1\geq2\sqrt{1+4 d\,\alpha_0\|u_0^{(N)}\|_{L^{\infty}}} .
\end{align*}
Define $\Gamma = \sqrt{d}\int_{\mathbb{R}^d}\tilde\varphi(\xi)\,(\xi\cdot e)^2 \,\mathrm d\xi$. Since $\log\frac{1+x}{1-x}\leq\frac{2x}{1-x}$, after replacing $\tilde\varphi$ by a sufficiently small positive multiple, the estimate of Theorem \ref{thm;blowup} yields
\begin{align*}
    T &\leq \frac{1}{\sqrt{1+4d\alpha_0\|u_0^{(N)}\|_{L^\infty}}}\log\left(\frac{\big(2\alpha_0 \Gamma(\log N)^{-1}\underset{n=2}{\overset{N}{\sum}} n^{-\sigma}-1\big)+\sqrt{1+4d\alpha_0\|u_0^{(N)}\|_{L^\infty}}}{\big(2\alpha_0 \Gamma(\log N)^{-1}\underset{n=2}{\overset{N}{\sum}} n^{-\sigma}-1\big)-\sqrt{1+4d\alpha_0\|u_0^{(N)}\|_{L^\infty}}}\right) \\
&\leq \frac{2}{\big(2\alpha_0 \Gamma(\log N)^{-1}\underset{n=2}{\overset{N}{\sum}} n^{-\sigma}-1\big)-\sqrt{1+4d\alpha_0\|u_0^{(N)}\|_{L^\infty}}} \\
&\leq \frac{4(1-\sigma)}{\alpha_0\Gamma}\frac{\log N}{(N+1)^{1-\sigma}-2^{1-\sigma}} \\
&\leq \frac{1}{\log N}.
\end{align*}

By the embedding $B^{1+\frac{d}{p}}_{p,r}\hookrightarrow B^{1}_{\infty,r}\hookrightarrow
B^{1}_{\infty,\infty}$ and Proposition \ref{prop2.1}, we have
\begin{align*}
    \|\nabla u(t)\|_{B^{0}_{\infty,\infty}}\leq C\,\|u(t)\|_{B^{1+\frac{d}{p}}_{p,r}},
    \quad t\in(0,T).
\end{align*}
Since $T<+\infty$, $\gamma(0)\rightarrow-\infty$ as $N\rightarrow+\infty$, applying Proposition \ref{thm;2} to the smooth solution and using the above embedding, we obtain
\begin{align*}
\int_0^T \|u(t)\|_{B^{1+\frac{d}{p}}_{p,r}} dt=+\infty.
\end{align*}
If $\|u(t)\|_{B_{p,r}^{1+\frac{d}{p}}}$ belongs to $L^\infty(0,T)$, then $\int_{0}^{T} \|u(t)\|_{B_{p,r}^{1+\frac{d}{p}}} \,\mathrm{d}t \leq T\bigl\|\|u(t)\|_{B_{p,r}^{1+\frac{d}{p}}}\bigr\|_{L^\infty(0,T)} < \infty$, which leads to a contradiction. In particular, for every sufficiently large $N$, there exists $t_N \in (0,T)$ satisfying
\begin{align*}
\|u(t_N)\|_{B_{p,r}^{1+\frac{d}{p}}} \geq \log N \quad \text{with} \quad t_N < T \leq \frac{1}{\log N}.
\end{align*}
Thus the solution cannot be extended continuously in $B_{p,r}^{1+\frac{d}{p}}$ up to $T$.
The proof of Theorem~\ref{thm;6} is complete.
\end{proof}

\smallskip
\noindent\textbf{Acknowledgments}

\noindent
This work is supported by the National Natural Science Foundation of China
(No.~12571261).
	
	\smallskip
	\noindent\textbf{Conflict of interest}
	The authors declare that they have no conflict of interest.
	
	\smallskip
	\noindent\textbf{Data Availability}
	No data was used for the research described in the article.

	\phantomsection
\addcontentsline{toc}{section}{\refname}

\end{document}